\documentclass[reqno,11pt]{amsart}
\usepackage{mathrsfs,graphicx,amsthm,amsfonts,amsxtra,xspace,array,cite,epic,float,ifpdf,bm,xcolor}
\usepackage[left=2.5cm]{geometry}
\usepackage{url}
\usepackage{enumerate}
\usepackage{amsmath,amssymb,amscd}
\usepackage{ifthen}     %%%%%   to enable swithcing between choices %%%

\newcommand\pubpri[2]{%
\ifthenelse{\equal{\version}{public}}%
{{#1}}%
{\marginpar{\scshape\small Pubpri Alert}{#2}}}
\newcommand\pubprinoalert[2]{%
\ifthenelse{\equal{\version}{public}}%
{{#1}}%
{#2}}
\newcommand\ignore[1]{}
\usepackage{color}
\providecommand\wantcolor{yes}   %  
\definecolor{backgroundyellow}{cmyk}{.2,.1,.8,.2}
\definecolor{backgroundblue}{rgb}{0,0,1}
\definecolor{backgroundred}{rgb}{1,0,0}
\definecolor{backgroundmagenta}{cmyk}{0,1,0,0}
\ifthenelse{\equal{\wantcolor}{no}}{\renewcommand\color[1]{}}{}

\newcommand\mysubsubsection[1]{%
		\subsubsection{\sffamily\upshape\mdseries #1}}

\newcommand\mysss{\mysubsubsection}

\ifpdf
  \usepackage[
    pdftex,
    colorlinks,%
    linkcolor=blue,citecolor=red,urlcolor=blue,
    hyperindex,%
    plainpages=false,%
    bookmarksopen,%
    bookmarksnumbered%
  ]{hyperref} 
\else
  \usepackage{hyperref}
\fi
\newtheorem{annotation}{Annotation}[section]
\newtheorem{theorem}[annotation]{%\color{blue}
		Theorem}
\newtheorem{lemma}[annotation]{%\color{blue}
		Lemma}
\newtheorem{definition}[annotation]{%\color{magenta}
		Definition}
\newtheorem{corollary}[annotation]{%\color{blue}
		Corollary}
\newtheorem{proposition}[annotation]{%\color{cyan}
		Proposition}
\newtheorem{example}[annotation]{%\color{yellow}
		Example}
		
\newcommand\bexample{\begin{example}\begin{rm}}
\newcommand\eexample{\end{rm}\hfill$\Box$\end{example}}
\newtheorem{examplenobox}[annotation]{%\color{yellow}
		Example}
\newcommand\bexamplenobox{\begin{examplenobox}\begin{rm}}
\newcommand\eexamplenobox{\end{rm}\end{examplenobox}}
\newtheorem{exercise}[annotation]{%\color{yellow}
		Exercise}
\newcommand\bexercise{\begin{exercise}\begin{rm}}
\newcommand\eexercise{\end{rm}\end{exercise}}
\newtheorem{notation}[annotation]{%\color{yellow}
		Notation}
\newcommand\bnotation{\begin{notation}\begin{rm}}
\newcommand\enotation{\end{rm}\end{notation}}

\newtheorem{remark}[annotation]{%\color{yellow}
		Remark}
\newcommand\bremark{\begin{remark}%\begin{mdseries}\begin{sffamily}
\begin{upshape}}
\newcommand\eremark{\end{upshape}%\end{sffamily}\end{mdseries}
\end{remark}}
\newcommand\bdefn{\begin{definition}%\begin{mdseries}\begin{sffamily}
\begin{upshape}}
\newcommand\edefn{\end{upshape}%\end{sffamily}\end{mdseries}
\end{definition}}
\newtheorem{caveat}[annotation]{%\color{yellow}
		Caveat}
\newcommand\bcaveat{\begin{caveat}%\begin{mdseries}\begin{sffamily}
\begin{upshape}}
\newcommand\ecaveat{\end{upshape}%\end{sffamily}\end{mdseries}
\end{caveat}}
\newenvironment{caveatstar}{%\color{yellow}
\par\noindent{\scshape\bfseries
  Caveat: }\begin{rm}}{\end{rm}\newline} 
\newcommand\bcaveatstar{\begin{caveatstar}}%\begin{mdseries}\begin{sffamily}
\newcommand\ecaveatstar{\end{caveatstar}}
\newenvironment{myproof}{%
\par\noindent{\scshape
  Proof: }\begin{rm}}{\hfill$\Box$\end{rm}\newline} 
\newcommand\bmyproof{\begin{myproof}}
\newcommand\emyproof{\end{myproof}}
\newenvironment{myproofnobox}{%
\par\noindent{\scshape Proof:}\begin{rm}}{\end{rm}\hfill\newline}
\newcommand\bmyproofnobox{\begin{myproofnobox}}
\newcommand\emyproofnobox{\end{myproofnobox}}
\newenvironment{solution}{%
\par\noindent{\scshape Solution: }\begin{rm}}{\hfill$\Box$\end{rm}\newline}
\newenvironment{solutionnobox}{%
\par\noindent{\scshape Solution: }\begin{rm}}{\end{rm}}
\newcommand\bsolution{\begin{solution}\begin{rm}}
\newcommand\esolution{\end{rm}\end{solution}}
\newcommand\bsolutionnobox{\begin{solutionnobox}\begin{rm}}
\newcommand\esolutionnobox{\end{rm}\end{solutionnobox}}
\newcommand\bthm{\begin{theorem}}
\newcommand\ethm{\end{theorem}}
\newcommand\bcor{\begin{corollary}}
\newcommand\ecor{\end{corollary}}
\newcommand\blemma{\begin{lemma}}
\newcommand\elemma{\end{lemma}}
\newcommand\bprop{\begin{proposition}}
\newcommand\eprop{\end{proposition}}
\newcommand\beqn{\begin{equation}}
\newcommand\eeqn{\end{equation}}
\newcommand\beqnstar{\begin{equation*}}
\newcommand\eeqnstar{\end{equation*}}
\newcommand\mtitle[1]%
{\marginpar{\sffamily\raggedright\upshape\small #1}}

\providecommand\finalized{no}
\ifthenelse{\equal{\finalized}{yes}}{}{\marginparwidth=2cm}
\ifthenelse{\equal{\finalized}{yes}}%
{}%
{}
\ifthenelse{\equal{\finalized}{yes}}%
{\newcommand\checked[1]{}}%
{\newcommand\checked[1]{\marginpar{[{\ttfamily\upshape\tiny CHECKED: #1}]}}}
\ifthenelse{\equal{\finalized}{yes}}%
{\newcommand\spellchecked[1]{}}%
{\newcommand\spellchecked[1]{\marginpar{[{\ttfamily\upshape\tiny SPELLCHECKED: #1}]}}}
\providecommand\version{public}   %  setting default  
\ifthenelse{\equal{\version}{public}}%
{\newcommand\mcomment[1]{}}%
{\newcommand\mcomment[1]{\marginpar{{\raggedright\sffamily\upshape\small
\begin{spacing}{0.75} #1\end{spacing}}}}}
\ifthenelse{\equal{\version}{public}}%
{\newcommand\fcomment[1]{}}%
{\newcommand\fcomment[1]{\footnote{#1}}}
\ifthenelse{\equal{\version}{public}}%
{\newcommand\comment[1]{}}%
{\newcommand\comment[1]{{\small #1}}}
\newcommand\fieldc{\mathbb{C}}

\newcommand{\be}{\begin{enumerate}}
\newcommand{\ee}{\end{enumerate}}
\newcommand{\beq}{\begin{equation}}
\newcommand{\eeq}{\end{equation}}
\newcommand{\beqs}{\begin{equation*}}                     
\newcommand{\eeqs}{\end{equation*}}
\newcommand{\complex}{\mathbb{C}}

\newcommand{\integers}{\mathbb{Z}}

\DeclareMathOperator{\tr}{trace}

\renewcommand\omega{\varpi}

\newcommand{\borel}{\widehat{\mathfrak{b}}}
\newcommand{\csa}{\mathfrak{h}}

\newcommand{\csaf}{\widehat{\mathfrak{h}}}
\newcommand\csahat{\widehat{\mathfrak{h}}}

\newcommand{\afw}{\widehat{W}} % the affine Weyl group
\newcommand{\digautos}{\Sigma} %dynkin diag autos

\newcommand{\afsl}{\widehat{\lieg}}

\newcommand\lieg{\mathfrak{g}}

\newcommand{\tsigma}{\tilde{\sigma}}

\newcommand{\bform}[2]{\left(#1 | #2\right)}
\newcommand{\cform}[2]{\langle #1,\, #2\rangle}
\newcommand{\f}{g}
\newcommand{\var}{\alpha_i^\vee t^{-j}, i\in I, j\in\mathbb{N}}
\newcommand{\q}{q}

\newcommand\lseqrpone{\lseq^{r+1}}

\newcommand\lrponeseq{\lseqrpone}
\newcommand\loneseq{\lseq^1}

\newcommand\barLam{\overline{\Lambda}_{i_\lambda}}
\newcommand{\Lamil}{\Lambda_{i_\lambda}}
\newcommand{\varpiil}{\varpi_{i_\lambda}}
\newcommand{\il}{i_{\lambda}}

\newcommand{\eoneseq}{\underline{\eta}^1}
\newcommand{\erseq}{\underline{\eta}^r}
\newcommand{\erponeseq}{\underline{\eta}^{r+1}}

\newcommand\lrseq{\lseq^r}

\newcommand\pattern{\mathcal{P}}

\newcommand\rpartd{\mathscr{P}_r(d)}
\newcommand{\lieh}{\csa}
\newcommand\pop{\mathfrak{P}}

\newcommand\pijiseqone{\underline{\pi(1)}^1}
\newcommand\pijiseqrr{\underline{\pi(r)}^r}
\newcommand\pijiseqtwo{\underline{\pi(2)}^i}

\newcommand\pijiseqoner{\underline{\pi(j)}^1}
\newcommand\pijiseqr{\underline{\pi (r)}^i}

\newcommand{\depthps}{d(\pop_s)}

\newcommand{\depthpspone}{d(\pop_{s+1})}

\newcommand{\depthpl}{d(\pop_\ell)}

\newcommand\lseq{\underline{\lambda}}

\newcommand\tseq{\underline{\theta}}

\newcommand{\pissseq}{\underline{\pi(s)}^{s}}
\newcommand{\pijsseq}{\underline{\pi(j)}^{s}}

\newcommand\pijiseq{\underline{\pi(j)}^i}

\newcommand\pijseq{\underline{\pi(j)}}

\newcommand\ljseq{\lseq^j}

\newcommand\plmkd{\popset_{\lambda,\mu}^k(d)}

\newcommand\popset{\mathbb{P}}

\newcommand\pkl{\popset_{\lambda+k\theta}}

\newcommand\pnotl{\popset_{\lambda}}

\newcommand\vcl[1]{v_{#1}}
\newcommand\rhocl[1]{\rho_{#1}}

\newcommand\currlieg{\lieg[t]}
\newcommand\curralg\currlieg

\newcommand\piseq{\underline{\pi}}

\newcommand{\dprimeij}{d^\prime_{i,j}}
\newcommand{\dss}{d_{s,s}}
\newcommand{\dprimess}{d^\prime_{s,s}}
\newcommand{\dsj}{d_{s,j}}
\newcommand{\dprimesj}{d^\prime_{s,j}}

\newcommand\dsq{d_{s,\q}}
\newcommand{\dprimesq}{d^\prime_{s,\q}}

\newcommand\pisqseq{\underline{\pi(\q)}^{s}}

\title[Stability of the Chari-Loktev bases]
{Stability of the Chari-Loktev bases \\for local Weyl modules of $\mathfrak{sl}_{r+1}[t]$}
\author{B.~Ravinder}
\address{Chennai Mathematical Institute\\ Plot H1, SIPCOT IT Park, Siruseri\\ Kelambakkam 603103,
India}
\email{bravinder@cmi.ac.in}
\subjclass[2010]{17B67 (17B10)}
\keywords{current algebra, Weyl module, Demazure module, Chari-Loktev basis, stability}

\begingroup\makeatletter\ifx\SetFigFont\undefined%
\gdef\SetFigFont#1#2#3#4#5{
  \reset@font\fontsize{#1}{#2pt}
  \fontfamily{#3}\fontseries{#4}\fontshape{#5}
  \selectfont}
\fi\endgroup
\begin{document}
\allowdisplaybreaks
\numberwithin{equation}{section}
\begin{abstract}
We prove stability of the Chari-Loktev bases with respect to the inclusions of local Weyl modules of the current algebra $\mathfrak{sl}_{r+1}[t]$.  
This is conjectured in \cite{RRV2} and the  $r=1$ case is proved in \cite{RRV1}.     
Local Weyl modules being known to be Demazure submodules in the level one representations of the affine Lie algebra $\widehat{\mathfrak{sl}_{r+1}}$,
we obtain, by passage to the direct limit,  bases for the level one representations themselves. 
\end{abstract}

\maketitle
\section{Introduction}
Let $\lieg$ be a finite-dimensional complex simple Lie algebra and $\lieg[t]=\lieg\otimes\mathbb{C}[t]$ be its current algebra. 
Local Weyl modules, introduced by Chari and Pressley  \cite{CP} are important
finite-dimensional $\lieg[t]$-modules.  These modules are characterized by the following universal property: 
any finite-dimensional $\lieg[t]$-module generated by a one-dimensional highest weight space, is a quotient of a local Weyl module.
Corresponding to every dominant integral weight $\lambda$ of~$\lieg$,  there is one local Weyl module  denoted by~$W(\lambda)$.

In \cite{CP}, for $\lieg=\mathfrak{sl}_2$, 
Chari and Pressley also produced  monomial bases for  local Weyl modules. Later
Chari and Loktev \cite{CL} extended the construction of these bases to $\lieg=\mathfrak{sl}_{r+1}$. 
Using  these bases they also showed that 
the local Weyl modules are  $\lieg[t]$-stable Demazure modules occurring in a level one representations of the affine Lie algebra $\widehat{\lieg}$. 
As a consequence,  we get an embedding of 
local Weyl modules $W(\lambda)\hookrightarrow W(\lambda+k\theta)$, where $\theta$ is the long root and $k$ is a 
non-negative integer. It is important to note that for every non-negative integer $k$, the local Weyl module $W(\lambda+k\theta)$ 
can be realized as $\lieg[t]$-stable Demazure module occurring in a 
fixed level one  representation of $\widehat{\lieg}$; we shall denote this level one  representation here by $V$. Thus we have a chain
of inclusions:
\beq\label{e:chainofw}
W(\lambda)\hookrightarrow W(\lambda+\theta)\hookrightarrow \cdots\hookrightarrow W(\lambda+k\theta)\hookrightarrow W(\lambda+(k+1)\theta)\hookrightarrow \cdots (\hookrightarrow V)
\eeq
such that the union of the modules in the chain equals $V$.

For $\lieg=\mathfrak{sl}_2,$ it is shown in \cite{RRV1} that after a suitable normalization, the Chari-Pressley bases   
behave well with respect to the inclusions in \eqref{e:chainofw}. Moreover in the limit,
these bases stabilize and give a nice monomial basis for $V$. For $\lieg=\mathfrak{sl}_{r+1}$, we consider the Chari-Loktev (CL) bases of local Weyl modules.
In \cite{RRV2}, an elegant combinatorial description for their parameterizing set is given: namely, as the
set of {\em partition overlaid patterns} (POPs). Moreover a weight preserving injective map  between the parameterizing sets of the
bases of $W(\lambda+k\theta)$ and $W(\lambda+(k+1)\theta)$ is given. Using this  it is conjectured that  after a suitable normalization the CL bases  also have the stability property  with respect to the inclusions in \eqref{e:chainofw}.
The purpose of this  paper is to prove this conjecture.

More precisely, let $\popset_\lambda$ denote the parametrizing set of
the CL basis for $W(\lambda)$: the elements of $\popset_\lambda$ are POPs with bounding sequence $\lseq$, where $\lseq$ is an integer sequence corresponding to $\lambda$.
In \cite{RRV2}, for each non-negative integer $k$, 
a weight preserving embedding from $\popset_{\lambda+k\theta}$ into $\popset_{\lambda+(k+1)\theta}$ is given. Thus we have a
chain $\popset_{\lambda}\hookrightarrow \popset_{\lambda+\theta}\hookrightarrow \popset_{\lambda+2\theta}\hookrightarrow\cdots$.
Given an element $\pop$ of $\popset_{\lambda}$ and a non-negative integer~$k$, let $\pop^k$ be its image
in $\popset_{\lambda+k\theta}$ and let $v_{\pop^k}$ be the corresponding normalized CL
basis element. Consider the sequence $v_{\pop^k}, k = 0, 1, 2,\ldots,$ of elements in $V$. We prove  that this sequence stabilizes for large k (see Theorem \ref{MT}).
Passing to the direct limit, we obtain
a basis for $V$ consisting of the stable CL basis elements (see \S \ref{ss:basesforllami}).

\subsection*{Acknowledgements}
The author acknowledges support from TIFR, Mumbai, under the Visiting Fellowship scheme.
\section{Notation and Preliminaries}\label{notn}
Throughout the paper, $\complex$  denotes the field of complex numbers, $\mathbb{Z}$ the set of integers,   
$\mathbb{N}$ the set of positive integers, $\mathbb{Z}_{\geq 0}$ the set of non-negative integers,   
$\complex[t]$ the polynomial ring,   $\complex[t,t^{-1}]$ the ring of Laurent polynomials, and $\mathbf{U}(\mathfrak{a})$ the universal
enveloping algebra corresponding to a complex Lie algebra $\mathfrak{a}$.
\subsection{The Lie algebra $\mathfrak{sl}_{r+1}$}
Let  $\lieg=\mathfrak{sl}_{r+1}$, the Lie algebra of $(r+1)\times(r+1)$ trace zero
matrices over the field~$\fieldc$ of complex numbers. 
Let 
$\csa$ be the {\em standard\/} Cartan subalgebra of $\lieg$ consisting of trace zero diagonal
matrices.  Let $R\subset\csa^*$ denote the root system of $\lieg$ with respect to $\csa$. For $\alpha\in R$, let $\lieg_\alpha$ be the root space corresponding to $\alpha$. 
Let 
$\mathfrak{b}$ be the {\em standard\/} Borel subalgebra of $\lieg$ consisting of upper triangular 
matrices. For 1$\leq i\leq r+1$, let $\varepsilon_i\in\csa^*$ be the projection to the $i^{th}$ 
co-ordinate. Set $I=\{1,2,\ldots,r\}$. 
Let $\alpha_i=\varepsilon_i-\varepsilon_{i+1}$, $i\in I$, be the set of simple roots 
and $\alpha_{i,j}=\alpha_i+\cdots+\alpha_j=\varepsilon_i-\varepsilon_{j+1}, \,1\leq i \leq j \leq r$, be the set $R^+$ of
positive roots of $\lieg$ with 
respect to $\mathfrak{b}$. Let $\theta=\alpha_{1,r}$ be the highest root of $\lieg$.
For $1\leq i,j \leq r+1$, let $E_{i,j}$ be the $(r+1)\times (r+1)$ matrix with 1  in the $(i,j)^{th}$
position
and 0 
elsewhere. Set 
\begin{align*}
& x_{\varepsilon_i-\varepsilon_j}=E_{i,j},  & h_{\varepsilon_i-\varepsilon_j}=E_{i,i}-E_{j,j},\qquad & \forall\,\, 1\leq i\neq j\leq r+1, \\
& x_{i,j}^+=x_{\alpha_{i,j}}^+=x_{\varepsilon_i-\varepsilon_{j+1}},  & x_{i,j}^-=x_{\alpha_{i,j}}^-=x_{\varepsilon_{j+1}-\varepsilon_i},  \qquad & \forall\,\,
1\leq i\leq j\leq r.
\end{align*}
Define subalgebras $\mathfrak{n}^{\pm}$ of $\lieg$ by 
$$\mathfrak{n}^{\pm}=\bigoplus_{1\leq i \leq j \leq r} \complex x_{{i,j}}^{\pm}.$$ Now we have the following decomposition: $\lieg= \mathfrak{n}^{-}\oplus \csa \oplus \mathfrak{n}^{+}.$ 
For $x,y\in\lieg$, let $(x|y):=\tr (xy)$ be the normalized invariant bilinear form on $\lieg$.
Let $W$ denote the Weyl group of $\lieg$.

Let $\omega_i=\varepsilon_1+\cdots+\varepsilon_i,\, i\in I,$ be the set of fundamental weights 
of $\lieg$.  
The weight lattice $P$, the set $P^+$ of dominant integral weights, and the root lattice $Q$ of $\lieg$ are defined as follows:
$$P=\sum_{i\in I}\mathbb{Z}\varpi_i, \quad P^+=\sum_{i\in I}\mathbb{Z}_{\geq0}\varpi_i, \quad\textup{and}\quad Q=\sum_{i\in I}\mathbb{Z}\alpha_i.$$
For $\lambda=m_1\varpi_1+\cdots+m_r\varpi_r\in P^+$, we associate an integer sequence
$\underline{\lambda}$ by $\underline{\lambda}=(\lambda_1\geq\cdots\geq \lambda_r\geq\lambda_{r+1}=0)$, where $\lambda_i:=m_i+\cdots+m_r$. Given an integer sequence 
$\underline{\lambda}=(\lambda_1\geq\cdots\geq \lambda_r\geq\lambda_{r+1}=0)$, we associate an element $\lambda$ of $P^+$ by $\lambda=\lambda_1\varepsilon_1+\cdots+\lambda_r\varepsilon_r=(\lambda_1-\lambda_2)\omega_1+\cdots+(\lambda_r-\lambda_{r+1})\omega_r$.
\subsection{The affine Lie algebra $\widehat{\lieg}$}
Let $\widehat{\lieg}$ be the (untwisted) affine Lie algebra corresponding to $\lieg$ defined by
\[\widehat{\lieg}=\lieg\otimes\complex[t,t^{-1}]\oplus\complex c\oplus \complex d,\]
where $c$ is central and the other Lie brackets are given by
\[[x\otimes t^m,y\otimes t^n]=[x,y]\otimes t^{m+n}+m\delta_{m,-n}(x|y)c,\]
\[[d,x\otimes t^m]=m(x\otimes t^m),\]
for all $x,y\in\lieg$ and integers $m,n$. The Lie subalgebras $\widehat{\csa}$ and $\widehat{\mathfrak{b}}$ of $\widehat{\lieg}$ are given by
\[\widehat{\csa}=\csa\oplus\complex c \oplus \complex d,\qquad \widehat{\mathfrak{b}}=\lieg\otimes t\complex[t]\oplus\mathfrak{b}\oplus\complex c\oplus\complex d.\]
We regard ${\csa}^{*}$ as a subspace of ${\widehat{\csa}}^{*}$ by setting
$\langle\lambda, c\rangle=\langle\lambda, d\rangle=0$ for   all $\lambda\in{\csa}^{*}$. 
Let $\delta, \Lambda_0\in{\csahat}^{*}$ be given by 
\[\langle\delta, \csa+\complex c\rangle=0, \, \, \langle\delta, d\rangle=1, \qquad \langle\Lambda_0, \csa+\complex d\rangle=0,
\,\,\langle\Lambda_0, c\rangle=1.\]
  There is a non-degenerate, symmetric, $\afw$-invariant, bilinear  form
$\bform{\cdot}{\cdot}$ on $\csaf^*$, given by requiring that $\csa^*$ be orthogonal to $\complex\delta + \complex \Lambda_0$,
together with the relations 
$$\bform{\alpha_i}{\alpha_i} =2, \quad\bform{\alpha_i}{\alpha_j} =-\delta_{j,i+1},\,\, \forall \,\,1\leq i < j \leq {r},\qquad 
\bform{\delta}{\delta} = \bform{\Lambda_0}{\Lambda_0} = 0,\,\,\textup{and}\,\,
\bform{\delta}{\Lambda_0}=1.$$ 

The elements $\alpha_0=\delta-\theta, \,\alpha_1,\ldots,\alpha_r$ are the simple roots of $\widehat{\lieg}$ 
and the corresponding coroots are $\alpha_0^\vee=c-h_\theta, \,\alpha_1^{\vee}=h_{\alpha_1},\ldots, \alpha_r^{\vee}=h_{\alpha_r}$. Set $\widehat{I}=I\cup\{0\}$. 
Let  $e_i$, $f_i$ ($i\in\widehat{I}$) denote the Chevalley generators of
$\afsl$: 
\[ e_0 = x_{1,r}^{-}\otimes t,\quad f_0 = x_{1,r}^{+}\otimes t^{-1},\qquad e_i = x_{i,i}^{+},\quad f_i = x_{i,i}^{-}, \quad\forall\,\, i\in I. \]
For $\alpha\in R$ and  $s\in\mathbb{Z}$, set
$ x_{\alpha+s\delta}=x_{\alpha}\otimes t^{s}.$
The weight lattice (resp.\ the set of dominant integral weights) of $\widehat{\lieg}$ is defined by 
\[ \widehat{P} \; (\text{resp.\ } \widehat{P}^+) = \{\Lambda \in
\csaf^*: \cform{\Lambda}{\alpha_p^\vee} \in \integers \; (\text{resp.\
}  \integers_{\geq 0}),\,\forall\, p\in\widehat{I}\}.\]
For an element $\Lambda\in \widehat{P}$, the  integer $\langle \Lambda,\,c\rangle$ is called 
the {\em level} of $\Lambda$. 
\subsection{The Weyl group of $\widehat{\lieg}$}
For each $p\in \widehat{I}$,  the fundamental reflection $s_{\alpha_p}$ (or $s_{p}$) is given by
\[s_{p}(\Lambda) = \Lambda-\langle \Lambda, \alpha_p^{\vee}\rangle \alpha_p,
 \quad\forall\,\,\Lambda\in\widehat{\csa}^*.\]
The subgroup $\widehat{W}$ of $GL(\widehat{\csa}^*)$ generated by all fundamental reflections 
$s_p,\,p\in\widehat{I}$ is called the affine Weyl group. 
We regard $W$ naturally as a subgroup of $\afw$. 
Given $\alpha \in \csa^*$, let $t_\alpha \in GL(\csaf^*)$ be defined by
\[
t_{\alpha}(\Lambda)=\Lambda+ \bform{\Lambda}{\delta} \alpha -\bform{\Lambda}{\alpha}
\delta - \frac{1}{2} \bform{\Lambda}{\delta} \bform{\alpha}{\alpha} \delta, \;\;
\text{ for } \Lambda \in \csaf^*.
\]
It is easy to see that
\beq\label{e:trp}
 t_\alpha\,t_\beta=t_{\alpha+\beta} \qquad\textup{and}\qquad w\,t_\alpha\, w^{-1}=t_{w\alpha}, \quad \forall\,\,\alpha,\beta\in\csa^*, w\in W.\eeq
The  translation subgroup $T_Q$ of $\afw$ is defined by $T_Q := \{t_{\alpha}\in 
GL(\csaf^*): \alpha \in Q\}$.

The following proposition gives the relation between $W$ and $\widehat{W}$. 
\begin{proposition}\cite[Proposition 6.5]{K}
$\widehat{W} = W\ltimes T_Q.$
\end{proposition}
The extended affine Weyl group $\widetilde{W}$ is the semi-direct product 
\[\widetilde{W}:=W\ltimes T_P,\]
where $T_P = \{t_{\alpha}\in GL(\csaf^*): \alpha \in P\}$.
For $i\in I$, consider the element $\sigma_i=t_{\omega_i}w_{0,i}w_0 \in \widetilde{W}$, where $w_0$ is the longest element in $W$ and $w_{0,i}$ is the longest element in $W_{\omega_i}$, the stabilizer of $\omega_i$ in $W$.
It is an automorphism of the Dynkin diagram of $\afsl$:
$$\sigma_i\,\alpha_p = \alpha_{i+p \,(\text{mod} \,\,r+1)},\quad\forall\,\,p\in \widehat{I},\qquad \textup{and} \qquad\sigma_i\,\rho=\rho.$$
Here, $\rho\in \csaf^*$ is the Weyl vector, defined by
$\cform{\rho}{\alpha^\vee_p} = 1,\,\, \forall\,\, p\in\widehat{I}$, and
$\cform{\rho}{d}=0$. 
Let $\digautos$ be the subgroup generated by $\{\sigma_i:  i\in I\}$. 
Now we also have
$\widetilde{W} = \afw \rtimes \digautos$ (see \cite[Chapter VI]{Bour}).
\subsection{Irreducible modules of $\widehat{\lieg}$}
Given $\Lambda \in \widehat{P}^{+},$ let $L(\Lambda)$ be the
irreducible  $\afsl$-module with highest weight $\Lambda$. It is the cyclic $\afsl$-module generated by $v_\Lambda$, with defining relations:
\begin{align}
h\,v_\Lambda &= \cform{\Lambda}{h} v_\Lambda, \quad\forall\, \,h \in \csaf,\label{e:defofllambd1}\\
e_p \,v_\Lambda &= 0, \quad\forall\,\,p\in \widehat{I},\label{e:defofllambd2}\\
f^{\cform{\Lambda}{\alpha^{\vee}_p}+1}_p \,v_\Lambda &=0,  \quad\forall\,\,p\in \widehat{I}.\label{e:defofllambd3}
\end{align}
\noindent
It has the weight space decomposition: $L(\Lambda) = \oplus_{\mu \in
  \csaf^*} L(\Lambda)_{\mu}$. The $\mu$ for which 
$L(\Lambda)_{\mu} \neq 0$ are the {\em weights of } $L(\Lambda)$.

The following two results are well-known:
\begin{proposition}\label{wtsbasicrep}\cite{K} Let $\Lambda\in \widehat{P}^+$ is of level 1. Then
\be
\item\label{i:weights} the set of weights of $L(\Lambda)$ is 
$\{t_{\alpha}(\Lambda)- m\delta\mid \alpha\in\,Q, 
m\in\integers_{\geq0}\}$,
\item for $\alpha\in\,Q$ and $m\in\integers_{\geq0}$, we have $$\dim\,L(\Lambda)_{t_{\alpha}(\Lambda)- m\delta}= \textup{the number of} \,\,r\textup{-colored partitions of}
\,\, m \,\,(\textup{see}\,\, \S \ref{ss:coloredp}).$$
\ee
\end{proposition}
\begin{theorem}\cite{FK}\label{weights}Given $m\in\mathbb{Z}_{\geq0}$, every element of  the weight space of $L(\Lambda_0)$ of weight $\Lambda_0-m\delta$ can be written as $\f_m\,v_{\Lambda_0}$ for some polynomial $\f_m$ in $\var$. 
\end{theorem}

We let $\Lambda_i := \sigma_i \Lambda_0$ for $i\in I$. Then, $\Lambda_0, \Lambda_1,\ldots, \Lambda_{r}$ are (a choice of) fundamental weights corresponding to the coroots $\alpha^{\vee}_0, \alpha^{\vee}_1,\ldots, \alpha_{r}^{\vee}$, i.e., $\cform{\Lambda_p}{\alpha^\vee_q} = \delta_{p,q}$ for  $p,q \in \widehat{I}$. 
 Let 
$v_{\Lambda_p}$ denote a highest weight vector of $L(\Lambda_p)$ for $p\in\widehat{I}$.
\subsection{The current algebra and its Weyl modules}
The current algebra
$\lieg[t]$:=$\lieg\otimes\complex [t]$ is a Lie
algebra with Lie bracket is obtained from that of $\lieg$ by
extension of scalars to $\complex[t]$:   
$$[x\otimes t^m, y\otimes t^n]:=[x,y]\otimes t^{m+n},\qquad \forall\,\,x, y \in \lieg,\,\,m, n\in \mathbb{Z}_{\geq0}.$$   
\bdefn 
(see~\cite[\S1.2.1]{CL})
Given $\lambda\in P^{+}$, the {\em local Weyl module\/} $W(\lambda)$ is the cyclic  $\lieg[t]$-module with generator $w_\lambda$ and relations:
$$ (\mathfrak{n}^{+}\otimes t\complex[t])\,w_{\lambda}=0, \qquad (h\otimes t^{s})\,w_{\lambda}=\langle\lambda, h\rangle\delta_{s,0},\,\,\forall\,\,h\in\csa, \,s\in\mathbb{Z}_{\geq0},\qquad 
f_{i}^{\langle\lambda, \alpha_i^{\vee}\rangle+1}\,w_{\lambda}=0,\,\,\forall\,\,i\in I.$$ 
\edefn
\subsection{Weyl modules as Demazure modules}Given $w\in\afw$ and $\Lambda\in\widehat{P}^{+}$, define a $\borel$-submodule $V_w(\Lambda)$  of $L(\Lambda)$ by 
$$V_w(\Lambda):={\mathbf{U}}(\borel)\,L(\Lambda)_{w\Lambda}.$$
 We call the $\borel$-module $V_w(\Lambda)$ as the 
{\em Demazure module} of $L(\Lambda)$ associated to $w$. More
generally,  given an element $w$ of the extended affine Weyl group
$\widetilde{W}$, we write $w = u \tau$ with $u \in \afw, \tau \in \digautos$, 
and define, following \cite{FoL},
the associated Demazure module by $V_{w}(\Lambda):=V_u\left(\tau(\Lambda)\right)$.

The following theorem identifies the local Weyl modules with the $\lieg[t]$-stable Demazure modules. 
\begin{theorem}\label{WvsD}\cite{CL, FoL}Given $\lambda\in P^+$, the local Weyl module $W(\lambda)$ 
is isomorphic to the $\lieg[t]$-stable Demazure module $V_{t_{w_0\lambda}}(\Lambda_0) $, as modules for the current algebra $\lieg[t]$.
\end{theorem}
\subsection{Inclusions of Weyl modules}\label{ss:inclweyl}
Let $\lambda\in P^+$ and
 $\lseq=(\lambda_1\geq\ldots\geq\lambda_r\geq\lambda_{r+1}=0)$ be its  corresponding sequence. Let $i_\lambda\in \widehat{I}$ be the remainder when $\sum_{i=1}^{r+1}\lambda_i$ is divided by $r+1$.
Set $\varpi_0$ as the zero element in $\csa^*$. It is easy to see that $\lambda-\varpi_{i_\lambda}\in Q$. 
  Since $w \Lambda_0=\Lambda_0$ for all $w\in W$, using \eqref{e:trp}, we have
 \begin{equation}
     t_{w_0\lambda}\Lambda_0=t_{w_0(\lambda-\varpiil)} t_{w_0\varpiil}(\Lambda_0)
   =t_{w_0(\lambda-\varpiil)}w_0\sigma_{\il}(\Lambda_0).
  \end{equation}
Thus from Theorem \ref{WvsD}, we get
\beq\label{e:wisoD}W(\lambda)\cong_{\lieg[t]} V_{t_{w_0(\lambda-\varpiil)}w_0}(\Lambda_{\il})\subset L(\Lambda_{\il}).\eeq
For every $k\geq0$, it is important to note that $i_{\lambda+k\theta}=i_\lambda$ and hence 
 $W(\lambda+k\theta)$ is also a Demazure submodule of $L(\Lamil)$.
 
For every $w\in W$, it is well-known that
\[t_{w_0(\lambda-\varpiil)}w\leq t_{w_0(\lambda+\theta-\varpiil)}w\leq \cdots\leq t_{w_0(\lambda+k\theta-\varpiil)}w\leq t_{w_0(\lambda+(k+1)\theta-\varpiil)}w\leq\cdots,\]
where $\leq$ is the  Bruhat order on the affine Weyl group. 
Hence using \eqref{e:wisoD}, we get  a chain of Demazure submodules of $L(\Lamil)$:
\beq\label{e:wmdir}W(\lambda)\hookrightarrow W(\lambda+\theta)\hookrightarrow\cdots\hookrightarrow W(\lambda+k\theta)\hookrightarrow W(\lambda+(k+1)\theta)
\hookrightarrow\cdots\quad(\hookrightarrow L(\Lamil))
\eeq
such that union of the modules in the  chain equals $L(\Lamil)$.
\subsection{Partitions} A {\em partition} is a non-increasing sequence of non-negative integers that is eventually
zero. 
 The non-zero elements of the sequence are called the {\em parts}
 of the partition. If the sum of the parts of a partition $\piseq:\pi_1\geq\pi_2\geq\cdots$ is $m$, then the partition is said to be a {\em partition of} $m$, and we write $|\piseq|=m.$
\subsubsection{Partition fits into a rectangle} Let $d, d^\prime$ be non-negative integers. We say that a partition {\em
 fits into a rectangle} $(d, d^\prime)$, if the number of parts is at most $d$ and every part is at most $d^\prime$. 
 \subsubsection{Complement of a partition} Let $\piseq$ be a partition fits in a rectangle $(d, d')$. The complement $\piseq^c$ of $\piseq$ is  given by 
 $(d'-\pi_d\geq\cdots\geq d'-\pi_1)$. Note that $\piseq^c$ also fits into the rectangle $(d, d').$
\subsection{Colored partitions}\label{ss:coloredp}
Let $r$ be a positive integer. An $r$-{\em colored partition} 
is a partition in which each part is assigned an integer between $1$ and $r$. The number
assigned to a part is its color. We may think of an $r$-colored partition as an ordered $r$-tuple 
$(\piseq^1,\ldots,\piseq^r)$ of partitions: the partition $\piseq^i$ consists of all parts of color $i$ of the $r$-colored partition.
An $r$-colored partition of a non-negative integer $m$ is an $r$-colored partition  $(\piseq^1,\ldots,\piseq^r)$ with $|\piseq^1|+\cdots+|\piseq^r|=m$.
\subsection{Gelfand-Tsetlin  patterns}\label{s:patterns}
A {\em Gelfand-Tsetlin (GT) pattern} (or just {\em pattern}) $\mathcal{P}$ is an array of
integral row vectors $\loneseq, \ldots, \lrseq$, $\lrponeseq$ (where 
${\underline{\lambda}}^j=(\lambda_1^j,\ldots,\lambda_j^j)$): 
\begin{align}
&\quad\qquad\qquad\qquad\qquad\lambda^1_1 \nonumber\\
&\quad\qquad\qquad\qquad\lambda^2_1\qquad\lambda^2_2\nonumber\\
&\quad\qquad\qquad\cdots\qquad\cdots\qquad\cdots\nonumber\\
&\qquad\qquad\lambda^{r}_1\qquad\qquad  \cdots\ \ 
\ \ \qquad\lambda^{r}_{r}\nonumber\\
&\qquad\lambda^{r+1}_1\qquad\lambda^{r+1}_2
\qquad \ \ \cdots \ \quad\qquad\lambda^{r+1}_{r+1}\nonumber
\end{align}
subject to the following conditions:
\[\lambda^{j+1}_i\geq\lambda^{j}_i\geq\lambda^{j+1}_{i+1},\quad\forall\,\,1\leq i\leq j\leq r.\]
The last sequence ${\lseq}^{r+1}$ of the pattern $\mathcal{P}$ is its
 {\em bounding sequence}.    
 \medskip
 
 Fix $\lambda\in P^+$ and a pattern $\pattern:\loneseq, \ldots, \lrseq$, $\lrponeseq=\lseq$ with bounding sequence $\lseq$.
\subsubsection{The weight of a pattern} The {\em weight} $\textup{wt}\,\pattern\in\csa^*$ of $\pattern$ is defined by
$$\textup{wt}\,\pattern:=a_1\varepsilon_1+a_2\varepsilon_2+\cdots+a_{r+1}\varepsilon_{r+1}, \quad\textup{where}\,\,a_j=\sum_{i=1}^{j}\lambda^j_i-\sum_{i=1}^{j-1}\lambda^{j-1}_i.$$
 Note that $a_1+a_2+\cdots+a_{r+1}=\lambda_1+\lambda_2+\cdots+\lambda_{r}.$ 
 \subsubsection{Differences of a pattern}For $1\leq i\leq j\leq r$, the {\em differences} $d_{i,j}(\pattern)$ and $d^\prime_{i,j}(\pattern)$ (or just $d_{i,j}$ and $d^\prime_{i,j}$ if  $\pattern$ is clear from the context)  of $\pattern$ are given by
$$d_{i,j}(\pattern):=\lambda_{i}^{j+1}-\lambda^j_{i} \qquad \textup{and} \qquad d^\prime_{i,j}(\pattern):=\lambda_{i}^j-\lambda^{j+1}_{i+1}.$$ 
\subsubsection{Area of a pattern} The {\em triangular area} or just {\em area} $\triangle(\pattern)$ of $\pattern$ is defined by 
 $$\triangle(\pattern):=\sum_{1\leq i\leq j\leq r} d_{i,j} d^\prime_{i,j}.$$
\subsubsection{Trapezoidal area of a pattern} The {\em trapezoidal area} $\square(\pattern)$ of  $\pattern$ is defined by 
 $$\square(\pattern):=\sum_{1\leq i\leq j\leq r}d_{i,j}(\sum_{p=i}^{j}d^\prime_{p,j}).$$
 \subsubsection{Shift of a pattern}
For $k\in\mathbb{Z}_{\geq0}$, the {\em shift}  $\pattern^k$ of $\pattern$ by $k$ is a pattern with 
bounding sequence $\lseq+k\tseq$: 
suppose that $\eoneseq$, \ldots, $\erseq$, $\erponeseq$ be the row vectors of $\pattern^k$, then
$$\eta^j_i:=\begin{cases} \lambda^j_i+2k, & i=1 \,\,\textup{and}\, \, 1<j\leq r+1,\\
\lambda^j_i, &  1<i=j\leq r+1,\\
\lambda^j_i+k, & \textup{otherwise}.\end{cases}$$
We observe that $$d_{i,j}(\pattern^k)=d_{i,j}(\pattern)+\delta_{i,j}k, \quad d^\prime_{i,j}(\pattern^k)= d^\prime_{i,j}(\pattern)+\delta_{1,i}k, \quad\forall\,\, 1\leq i\leq j\leq r, 
 \quad \textup{and}\quad\textup{wt}\,\pattern^k=\textup{wt}\,\pattern.$$
\subsection{Partition overlaid patterns (POPs)} \label{s:pops}
A {\em partition overlaid pattern (POP)} consists of a GT pattern $\pattern$, and for every pair $(i,j)$ of integers with $1\leq i\leq j\leq r$, 
a partition $\pijiseq$ that fits into the rectangle $(d_{i,j}(\pattern), d^\prime_{i,j}(\pattern))$ (see \cite{RRV2} for more details).
For $\lambda\in P^+$, let $\popset_{\lambda}$ denote the set  of POPs with bounding sequence $\lseq$. 

The {\em bounding sequence, area} $\triangle(\pop)$, {\em trapezoidal area} $\square(\pop)$, {\em weight} $\textup{wt}\,\pop$, and the {\em differences} $d_{i,j}(\pop), d^\prime_{i,j}(\pop)$ (or just $d_{i,j}$ and $d^\prime_{i,j}$ if  $\pop$ is clear from the context), $1\leq i\leq j\leq r)$, of a POP $\pop$ are just the corresponding notions attached to the underlying pattern.
\medskip

Fix $\lambda\in P^+$ and a POP $\pop$ with bounding sequence $\lseq$.  Let
$\loneseq, \ldots, \lrseq$, $\lrponeseq=\lseq$ be the underlying pattern of $\pop$  and $\pijiseq$, $1\leq i\leq j\leq r$, be the partition
overlay. 
\subsubsection{Restriction of a POP}
For $1\leq i\leq j\leq r+1$,  
define  $\underline{\lambda^j_i}:= \lambda^{j}_i, \lambda^{j}_{i+1}, \ldots, \lambda^{j}_{j}$. 
Observe that $\underline{\lambda^j_1}=\ljseq$.
For $s\in I\cup\{r+1\}$, the {\em restriction}  $\pop_s$ or $\textup{res}_s(\pop)$ of $\pop$ to $s$ is a POP  with bounding sequence 
$\underline{{\lambda}^{r+1}_s}$: the row vectors of $\pop_s$ are $\underline{\lambda^s_s}, \underline{\lambda^{s+1}_s}, \ldots, \underline{\lambda^{r+1}_s}$ and $\pijiseq$, $s\leq i\leq j\leq r$, be the partition
overlay. Observe that $\pop_1=\pop$.  
\subsubsection{Depth of a POP}
The {\em depth} $d(\pop)$ of $\pop$ is defined by
$$d(\pop):=\sum_{1\leq i\leq j\leq r} d^j_i(\pop), \qquad\textup{where} \qquad  d^j_i(\pop):= d_{i,j}(\sum_{p=i+1}^{j}d^\prime_{p,j})+|\pijiseq|.$$
 For $s\in I$, we observe that 
\beq\label{e:depthofps}
d(\pop_s)=\sum_{s\leq i\leq j\leq r} d^j_i(\pop)=
d(\pop_{s+1})+
\sum_{j=s}^r{d^j_s(\pop)}=d(\pop_{s+1})+\sum_{j=s+1}^r{d^j_s(\pop)}+|\pissseq|.\eeq
From \cite[Corollary 3.4]{RRV2}, we have the following:
\beq\label{e:maxarea}
\square(\pop)=\triangle(\pop)+d(\pop)-\sum_{1\leq i\leq j\leq r}|\pijiseq|=\frac{1}{2}\big((\lambda|\lambda)-({\textup{wt}}\,\pop|{\textup{wt}}\,\pop)\big).
\eeq
\subsubsection{Shift of a POP}
For $k\in\mathbb{Z}_{\geq0}$, the {\em shift}  $\pop^k$ of $\pop$ by $k$ is a POP with 
bounding sequence $\lseq+k\tseq$: the underlying pattern of $\pop^k$ is $\pattern^k$ and $\pijiseq$, $1\leq i\leq j\leq r$, be the partition
overlay. Note that the underlying partition overlay for $\pop^k$ and $\pop$ is same.
It is easy to observe that 
\beq\label{e:wtd}
\textup{wt}\,\pop^k=\textup{wt}\,\pop \qquad\textup{and}\qquad d(\pop^k)=d(\pop).\eeq
\subsubsection{Shift and then restrict}For $k\in\mathbb{Z}_{\geq0}$ and $s\in I\cup\{r+1\}$, 
set $\pop^k_s:=\textup{res}_s(\pop^k)$.
\subsubsection{Invariant set of a POP} The {\em invariant set}  $\mathcal{I}(\pop)$ of $\pop$ is a set consists of the  partition overlay of $\pop$ and 
the differences of $\pop$ which are invariant under the shift. More precisely,
$$\mathcal{I}(\pop):=\{d_{i,j}(\pop):1\leq i<j\leq r\}\cup\{ d^\prime_{i,j}(\pop):1<i\leq j\leq r\}\cup\{\pijseq^i:1\leq i\leq j\leq r\}.$$
For $s\in I\cup\{r+1\}$, note that 
$$\mathcal{I}(\pop_s)=\{d_{i,j}(\pop):s\leq i<j\leq r\}\cup\{d^\prime_{i,j}(\pop):s<i\leq j\leq r \}\cup\{\pijseq^i:s\leq i\leq j\leq r\}.$$
Set  $$\mathcal{I}^j_s(\pop):=\{ d_{s,j}(\pop), \pijsseq\}\cup\{\dprimeij(\pop):s<i\leq j\},\,\,1\leq s< j\leq r,\quad\textup{and}\quad \mathcal{I}^s_s(\pop):=\{ \underline{\pi(s)}^s\}, \,\, s\in I.$$
Now we have
\beq\label{e:invsetofps}
\mathcal{I}(\pop_s)=\mathcal{I}(\pop_{s+1})\cup\big(\cup_{s\leq j\leq r}\mathcal{I}^j_s(\pop)\big), \quad\forall\,\,s\in I.
\eeq

\section{The main result}
\subsection{The Chari-Loktev bases for local Weyl modules in type A}\label{ss:clbase} 
In this subsection, we recall the bases given by Chari and Loktev  \cite{CL} in terms of POPs (see \cite{RRV2}). 
Fix notation and terminology as in~\S\ref{notn}.

\subsubsection{}\label{sss:cl}   
Let $d$, $d^\prime$ be non-negative integers and $\piseq$ be a partition that fits into
the rectangle $(d,d^\prime)$. For $\alpha\in R^+$, the monomial $x^\pm_{\alpha}(d,\,d',\,\piseq)$ corresponding to the complement of $\piseq$ is given by
$$x^\pm_{\alpha}(d,\,d',\,\piseq):= \big(\prod_{i=1}^{d}   x^\pm_\alpha\otimes t^{d^\prime-\pi_i}\big).$$
Set $x_{i,j}^\pm(d,\,d',\,\piseq):=x_{\alpha_{i,j}}^\pm(d,\,d',\,\piseq)$ for all $1\leq i\leq j\leq r$.
\subsubsection{}
Let $\lambda\in P^+$ and ~$\pop$ be a POP with bounding sequence $\lseq$.  Let
$d_{i,j},\, d'_{i,j}$, $1\leq i\leq j\leq r$,  be the differences  and $\pijiseq$, $1\leq i\leq j\leq r$, be the partition
overlay of $\pop$.    
Define $\rhocl{\pop}\in \mathbf{U}(\mathfrak{n}^-\otimes\complex[t])$ as follows:   
\begin{equation}\label{e:clmonom} \rhocl{\pop}:=
x_{1,1}^-(d_{1,1}, \,d^\prime_{1,1},\, \pijiseqone)\,
\big(\prod_{i=1}^2 x_{i,2}^-(d_{i,2},\, d^\prime_{i,2}, \,\pijiseqtwo)\big)\,\cdots\,
\big(\prod_{i=1}^r x_{i,r}^-(d_{i,r}, \,d^\prime_{i,r},\, \pijiseqr)\big).
\end{equation}
The order of the factors matters in the expression for $\rhocl{\pop}$. 
Since $[x^-_{i,j}, x^-_{p,q}]=0, \forall\,\,1\leq i\leq p\leq q\leq j\leq r$, it is easy to see that
\begin{equation}\label{e:clmonom2}
 \rhocl{\pop}=\big(\prod_{j=1}^r x_{1,j}^-(d_{1,j}, \, d^\prime_{1,j},\, \pijiseqoner)\big)\,
\big(\prod_{j=2}^r x_{{2}, j}^-(d_{2, j}, \, d^\prime_{2, j}, \,\underline{\pi(j)}^2)\big)\,\cdots\,
x_{r,r}^-(d_{r,r}, \, d^\prime_{r,r},\, \pijiseqrr).
\end{equation} Set $\rho_{\pop_{r+1}}:=1.$
We observe that $$\rho_{\pop_s}=\big(\prod_{j=s}^r x_{{s}, j}^-(d_{s, j}, \, d^\prime_{s, j},\, \pijsseq)\big)\,\rho_{\pop_{s+1}},\quad\,\,\forall \,\,s\in I.$$
Define  $\vcl{\pop}:=\epsilon_\pop\,\rhocl{\pop} \,w_\lambda,$ where $\epsilon_\pop\in\{\pm1\}$ is defined in \S\ref{ss:MTp}.

The following theorem is proved in \cite{CL} (see \cite[Theorem 4.5]{RRV2} for the current formulation).
 \begin{theorem}\label{t:cl}\cite{CL,RRV2}
   The elements
$v_\pop$,   $\pop$ belongs to the set $\popset_\lambda$ of POPs with
bounding sequence~$\lseq$,  form a basis for the local Weyl module~$W(\lambda)$.
 \end{theorem}
We shall call the bases given in the last theorem as the {\em Chari-Loktev (or CL)} bases.

\subsection{The main theorem: stability of the CL bases} We wish to study for $\lambda\in P^+$ and $k\in\mathbb{Z}_{\geq0}$, the compatibility of CL bases with respect to the embeddings 
$W(\lambda)\hookrightarrow W(\lambda+k\theta)$ in $L(\Lambda_{i_\lambda})$ (see \S \ref{ss:inclweyl}). 
We first recall the weight preserving embedding from $\pnotl$ into $\pkl$ given in \cite[Corollary~5.13]{RRV2} at the level of the parametrizing sets of these bases:
for ~$\pop\in\pnotl$, the shift $\pop^k$ of $\pop$ by $k$ be its image in $\pkl$.

For every $\lambda\in P^+$, we will fix the following choice of $w_\lambda$ in $L(\Lambda_{i_\lambda})$:
$$w_\lambda:=T_{\lambda}\,v_{\Lambda_0},$$
where $T_\lambda$ is the linear isomorphism from $L(\Lambda_0)\rightarrow L(\Lambda_{i_\lambda})$ defined in \S \ref{ss:tlamdef}.
\begin{lemma}\label{l:wt}
 Let $\lambda\in P^+$ and $\pop\in\pnotl$.
 Then
 $\textup {the weight of}\,\, \vcl{\pop} \,\,\textup{in} \,\,L(\Lambda_{i_\lambda})\,\, \textup{is}$ $$t_{\textup{wt}\,\pop-\barLam}(\Lambda_{i_\lambda})-d(\pop)\delta, $$ 
 where $\barLam$ denotes the restriction of $\Lamil$ to $\lieh$ . 
\end{lemma}
\begin{proof}It is clear from the definition of $\vcl{\pop}$ that  its weight in $L(\Lambda_{i_\lambda})$ is 
\begin{align}
&t_{\lambda}(\Lambda_0)-\sum_{1\leq i\leq j\leq r}d_{i,j}\alpha_{i,j}+\big(\triangle(\pop)-\sum_{1\leq i\leq j\leq r}|\pijiseq|\big)\delta\nonumber\\
&=\Lambda_0+{\textup{wt}}\,\pop-\big(\frac{1}{2}(\lambda|\lambda)-\triangle(\pop)+\sum_{1\leq i\leq j\leq r}|\pijiseq|\big)\delta\nonumber\\
&=\Lambda_0+\textup{wt}\,\pop-\big(\frac{1}{2}({\textup{wt}}\,\pop|{\textup{wt}}\,\pop)+d(\pop)\big)\delta \label{e:lemwt1}
\end{align}
where the last equality follows from \eqref{e:maxarea}.
 Since $\Lambda_{i_\lambda}$ is of level~$1$,  we obtain using \cite[(6.5.3)]{K} that
\beq\label{e:tgL2}
t_{\textup{wt}\,\pop-\barLam}(\Lamil)=\Lambda_0+{\textup{wt}}\,\pop+\frac{1}{2}((\Lamil|\Lamil)-({\textup{wt}}\,\pop|{\textup{wt}}\,\pop))\delta.
\eeq
Since $(\Lamil|\Lamil)=0$, we get the result from \eqref{e:lemwt1}--\eqref{e:tgL2}.
\end{proof}
The following is immediate from Lemma \ref{l:wt} and \eqref{e:wtd}.
\begin{lemma}Let $\lambda\in P^+$, $\pop\in\pnotl$, and $k\in\mathbb{Z}_{\geq0}$. Then the basis vectors $\vcl{\pop}\in W(\lambda)$ and  $\vcl{\pop^k}\in W(\lambda+k\theta)$ lie in the same weight space of $L(\Lambda_{i_\lambda})$.
\end{lemma}
It is not true that $\vcl{\pop}$ and $\vcl{\pop^k}$ are equal as elements of $L(\Lambda_{i_\lambda})$ (see \cite[Example 1]{RRV1}).
We will however see below that $\vcl{\pop}=\vcl{\pop^k}$ for all {\em stable} $\pop$. More precisely, let
$$\mathbb{P}^{\textup{stab}}(\lambda):=\{\pop\in\pnotl:d_{\ell,\ell}(\pop)\geq\depthpl,\,\, \forall \,\,1\leq \ell\leq r\} \,\,(\textup{see}\,\,\S\S\ref{s:patterns}-\ref{s:pops}).$$

The following theorem is the main result of this paper. 
\begin{theorem}\label{MT} Let $\lieg=\mathfrak{sl}_{r+1}$. 
Let $\lambda\in P^+$ and $\pop\in\mathbb{P}^{\textup{stab}}(\lambda)$. Then
 $$\vcl{\pop^k}=\vcl{\pop} \qquad\textup{for all}\,\, k\in\mathbb{Z}_{\geq0},$$ 
  i.e., they are equal as elements of $L(\Lambda_{i_\lambda})$.
 \end{theorem}
This theorem is proved in \S\ref{pf:MT}.
\begin{remark}
 Thorem \ref{MT} is conjectured in \cite[Conjecture~6.1]{RRV2} and the $r=1$ case is proved in \cite[Theorem~6]{RRV1} under the additional assumption that 
 $$d(\pop)\leq\begin{cases}
                \textup{min}\{d_{1,1}(\pop),\, d^\prime_{1,1}(\pop)\}, &\lambda_1 \,\,\textup{even},\\
                \textup{min}\{d_{1,1}(\pop),\, d^\prime_{1,1}(\pop)-1\}, &\lambda_1\,\, \textup{odd}.
                               \end{cases}$$
\end{remark}
\subsection{Bases for level one representations of  $\widehat{\lieg}$}\label{ss:basesforllami}
Fix $i\in\widehat{I}$, $\gamma\in Q$, and $d\in\mathbb{Z}_{\geq0}$. Consider the irreducible module $L(\Lambda_i)$ and its weight space of weight
 $t_{\gamma}(\Lambda_i)-d\delta$.
Set $\mu=\varpi_i+\gamma$, the restriction of  $t_{\gamma}(\Lambda_i)-d\delta$ to $\csa^*$.
Let $\lambda\in P^+$ such that $\mu$ is a weight of the corresponding irreducible representation $V(\lambda)$ of $\lieg$.
Note that $i_\lambda=i$. 

For $k\in\mathbb{Z}_{\geq0}$, from Lemma \ref{l:wt},  we get that
the CL basis indexing set for $W(\lambda+k\theta)_{t_{\gamma}(\Lambda_i)-d\delta}$ is the set $\plmkd$ of POPs with
bounding sequence $\lseq+k\tseq$ with weight $\mu$ and depth $d$.  
From \cite[Theorem~5.10]{RRV2}, for $k\geq d$, there exist a bijection from the set $\rpartd$ of all $r$-colored
partitions of $d$  onto $\plmkd$. Since this bijection is produced by the ``shift by $k$'' operator, we have
\beq\label{e:cruforstabbasis}
d_{\ell,\ell}(\pop)\geq k,\quad\forall\,\, 1\leq \ell\leq r, \qquad \textup{for every}\,\,\pop\in\plmkd.
\eeq
For $k\geq d$, 
 by Proposition \ref{wtsbasicrep}, we now have
$$W(\lambda+k\theta)_{t_{\gamma}(\Lambda_i)-d\delta}= L(\Lambda_i)_{t_{\gamma}(\Lambda_i)-d\delta},$$
and the set $\mathcal{B}_{\gamma, d}:=\{v_\pop:\pop\in\plmkd\}$ is a basis for $L(\Lambda_i)_{t_{\gamma}(\Lambda_i)-d\delta}$.
By Theorem \ref{MT}, using \eqref{e:cruforstabbasis}, the set $\mathcal{B}_{\gamma, d}$ is independent of the choice of $k$  for any  $k\geq d$. 

Finally, to obtain a basis for $L(\Lambda_i)$, we take the disjoint union over the weights of $L(\Lambda_i)$:
$$\mathcal{B}:=\bigsqcup_{\gamma, d}\mathcal{B}_{\gamma, d}.$$
We may view $\mathcal{B}$ as a direct limit of the CL bases for the Demazure submodules of $L(\Lambda_i)$.

\section{Proof of the main result}\label{s:proof}
\subsection{Frenkel-Kac translation operators}
 We recall the necessary facts from \cite{FK}.
 Let $(V,\pi)$ be an integrable representation of $\afsl$ with weight space decomposition $V = \oplus_{\nu \in \csaf^*} V_{\nu}$. 
 For a real root $\gamma = \alpha + s \delta \;(\alpha \in R,\, s \in \integers)$ of $\afsl$, we define
 \begin{equation}
 r^\pi_{\gamma}:=e^{-\pi(x_{\gamma})} e^{\pi(x_{-\gamma})} e^{-\pi(x_{\gamma})}.
 \end{equation}
The operator  $r^\pi_{\gamma}$ is a linear automorphism of $V$ such that $r^\pi_{\gamma}(V_{\nu})=V_{s_{\gamma}(\nu)}$, where $s_\gamma \in 
\widehat{W}$ is the reflection defined by $\gamma$. 
Given $w\in \widehat{W}$ and its reduced expression $w=s_{i_1}s_{i_2}\cdots s_{i_q}$, define
$$r^\pi_w:=r^\pi_{\alpha_{i_1}}r^\pi_{\alpha_{i_2}}\cdots r^\pi_{\alpha_{i_q}}.$$
Note that $r^\pi_w(V_\nu)=V_{w\nu}$.

For each $\beta\in Q$, there exists a translation operator $T^\pi_\beta$ on $V$ such that
 \begin{align*}
  T^\pi_{\alpha}=r_{\delta-\alpha}\,r_{\alpha},\quad\alpha\in R\qquad {and}\qquad
  T^\pi_\beta\,T^\pi_{\beta^\prime}=\epsilon(\beta,\,\beta^\prime) \,T^\pi_{\beta+\beta^\prime},\quad\beta,\beta^\prime\in Q,
 \end{align*}
where $\epsilon$ is a 2-cocycle of $Q$ with values in $\{\pm1\}$ (see \cite[\S 2.3]{FK}).
These operators satisfy $T^\pi_\beta(V_\nu) = V_{t_\beta(\nu)}$ for all $\nu \in \csaf^*$, $\beta \in Q$.

We will only need these operators in two cases, namely when $(V, \pi)$ is either the adjoint representation 
or the basic representation of $\afsl$. We note that $T_\beta^{\mathrm{ad}}$ is in fact a 
Lie algebra automorphism of $\afsl$. For ease of notation, we will denote the translation operators corresponding to the basic representation simply by $T_\beta$, 
suppressing the $\pi$ in the superscript.

The key properties of the translation operators are given in \cite[Propositions 1.2 and 2.3]{FK}.
We summarize them for our context below:
\begin{proposition}\label{fkprop}\cite{FK} Let $\mu\in Q$.  Then
\begin{enumerate}
\item\label{fkprop:p3}
 $T_{\mu}\,T_{-\mu}=\textup{id}_{L(\Lambda_0)}.$
 \smallskip
 \item\label{fkprop:p2} $T_{\mu-d\alpha}\, T_{d\alpha}=\epsilon(\mu-d\alpha,\,d\alpha)\, T_{\mu},\quad\forall \,\,\alpha\in R^+,\, d\in \mathbb{Z}_{\geq0}$.
 \smallskip
 \item $T_\mu\,X\,T_{-\mu}\,v=T^{\textup{ad}}_\mu(X)\,v,\quad\forall\,\,X\in\afsl, \,v\in L(\Lambda_0).$
 \smallskip
 \item \label{fkprop:p4} $T_{\mu}^{\textup{ad}}(x^-_{\alpha}\otimes t^s)=(x^-_{\alpha}\otimes t^{s+(\mu|\alpha)}),\quad\forall\,\,\alpha\in R^+,\, \, s\in \mathbb{Z}.$
 \smallskip
 \item \label{fkprop:p1}$T_{\mu}\,(h\otimes t^s)\,v=T^{\textup{ad}}_\mu(h\otimes t^s)\,T_\mu\,v=
(h\otimes t^s)\,T_{\mu}\, v, \quad\forall\,\,h\in\csa,\, v\in L(\Lambda_0), \,s\in \mathbb{Z} \setminus \{0\}$.
\end{enumerate}
   \end{proposition}
\subsection{}The goal of this subsection is to define a  translation operator $T_{\varpi_i}$ associated to a fundamental weight $\varpi_i\,\, (i\in I)$ of $\lieg$.
\subsubsection{}
Let $\tau$ be an automorphism of $\afsl$ such that $\tau \csaf = \csaf$. We have the induced action of $\tau$ on $\csaf^*$ by $\cform{\tau\lambda}{h} = \cform{\lambda}{\tau^{-1}h}$. 
Given an $\afsl$-module $V$, let $V^{\tau}$ denote the module with the twisted action $$x\circ v=\tau^{-1}(x)\,v, \qquad\text{ for }\qquad x\in \afsl, \,v \in V.$$
Observe that for automorphisms $\tau_1, \tau_2$, we have $V^{\tau_1\tau_2} \simeq \left(V^{\tau_2}\right)^{\tau_1}$.

For $i\in I$, we now study the twisted actions on $L(\Lambda_0)$ by two specific automorphisms $\tsigma_i, \, \tilde{\phi}_{w_0w_{0,i}}$ of $\afsl$. 
First, recall from \S\ref{notn} that $\sigma_i=t_{\omega_i}w_{0,i}w_0$ is an automorphism of the Dynkin diagram of $\afsl$:
$$\sigma_i\,\alpha_p = \alpha_{i+p \,(\text{mod} \,\,r+1)},\quad\forall\,\,p\in \widehat{I},\qquad \textup{and} \qquad \sigma_i\,\rho=\rho.$$
 Consider the Lie algebra automorphism $\tsigma_i$ of $\afsl$ given by the relations
\beq\label{e:tsigma}
\tsigma_i(e_p) = e_{i+p\, (\text{mod} \,\,r+1)}, \,\,\tsigma_i(f_p) = f_{i+p \,(\text{mod} \,\,r+1)}, \,\, \tsigma_i(\alpha^\vee_p) = \alpha^\vee_ {i+p\, (\text{mod}\,\,r+1)},\,\,\forall\,\, p\in\widehat{I}, \,\,\text{and} \,\,\tsigma_i(\rho^{\vee}) = \rho^\vee.
\eeq
Here
 $\rho^\vee \in \csaf$ is the unique element for which  $\langle\alpha_p,{\rho}^{\vee}\rangle=1,\,\,\forall\,\, p\in\widehat{I},$ and $\langle\Lambda_0,\rho^{\vee}\rangle=0$. Observe that 
 $\tsigma_i$  leaves $\csaf$ invariant, and its induced action on $\csaf^*$ coincides with $\sigma_i$,  i.e.,
\beq\label{e:coincide1} 
\cform{\sigma_i\Lambda}{h} = \cform{\Lambda}{\tsigma_i^{-1}h}, \qquad \forall\,\,h\in\widehat{\csa},\, \Lambda\in \widehat{\csa}^*.
\eeq 

Given $w\in W$, define the map $\phi_w: \lieg\rightarrow\lieg$ by
$$\phi_w(x_\alpha)=x_{w\alpha},\quad \phi_w(h_\alpha)=h_{w\alpha},\qquad\forall\,\,\alpha\in R.$$
It is easy to see that $\phi_w$ is an automorphism of $\lieg$
and it can be extended to an automorphism $\tilde{\phi}_w$ of $\afsl$ by defining 
$$\tilde{\phi}_w(c)=c, \quad\tilde{\phi}_w(d)=d,\qquad \tilde{\phi}_w(x\otimes t^s)={\phi}_w(x)\otimes t^s, \quad \forall \,\,x\in\lieg,\,s\in\mathbb{Z}.$$
Clearly $\tilde{\phi}_w$  leaves $\csaf$ invariant, and its induced action on $\csaf^*$ coincides with $w$,  i.e.,
\beq\label{e:coincide2}
 \cform{w\,\Lambda}{h} = \cform{\Lambda}{\tilde{\phi}_{w}^{-1}h}, \qquad \forall\,\,h\in\widehat{\csa},\, \Lambda\in \widehat{\csa}^*.
\eeq

For $i\in I$, set ${\tilde{t}_{\varpi_i}}:= \tsigma_i \,\tilde{\phi}_{w_0w_{0,i}}.$ Observe that ${\tilde{t}_{\varpi_i}}$ leaves $\csaf$ invariant, and ${\tilde{t}_{\varpi_i}}(\lieg_\alpha)=\lieg_{{{t}_{\varpi_i}}(\alpha)}$ for all $\alpha\in R$. 
From \eqref{e:coincide1}--\eqref{e:coincide2}, we have
\beq\label{e:coincidefort} 
\cform{t_{\varpi_i}\Lambda}{h} = \cform{\Lambda}{\tilde{t}_{\varpi_i}^{-1}h}, \qquad \forall\,\,h\in\widehat{\csa},\, \Lambda\in \widehat{\csa}^*.
\eeq
For $p\in\widehat{I}$, note that 
\beq\label{e:w0w0i}
{w_0w_{0,i}}(\alpha_p)=\begin{cases}
                      \alpha_{r+1+p-i}, & p<i,\\
                      \alpha_{p-i}, & p>i,\\
                      -\theta, & p=i,\\
                      \alpha_{r+1-i}+\delta, & p=0.\\
                     \end{cases}
\qquad\textup{and}\qquad                    
{w_{0,i}w_0}(\alpha_p)=\begin{cases}
                      \alpha_{p+i-r-1}, & r+1-p<i,\\
                      \alpha_{p+i}, & r+1-p>i,\\
                      -\theta, & r+1-p=i,\\
                      \alpha_{i}+\delta, & p=0.\\
                     \end{cases}
\eeq
Using \eqref{e:tsigma} and \eqref{e:w0w0i}, for $p\in\widehat{I}$, we get 
\beq\label{e:ttildwi}
\tilde{t}_{\varpi_i}(e_p)=\begin{cases}
                       e_p, &p\neq 0,i,\\
                       e_i\otimes t^{-1}, &p=i,\\
                       x_{-\theta}\otimes t^2, &p=0,
                     \end{cases}
\qquad
\tilde{t}_{\varpi_i}(f_p)=\begin{cases} f_p, &p\neq 0,i,\\
                       f_i\otimes t, &p=i,\\
                       x_{\theta}\otimes t^{-2}, &p=0,
                     \end{cases}
\eeq
and
\beq\label{e:ttildinvwi}
{\tilde{t}}_{\varpi_i}^{-1}(e_p)=\begin{cases} e_p, &p\neq 0,i,\\
                       e_i\otimes t, &p=i,\\
                       x_{-\theta}, &p=0,
                     \end{cases}
\qquad
{\tilde{t}}^{-1}_{\varpi_i}(f_p)=\begin{cases} f_p, &p\neq 0,i,\\
                       f_i\otimes t^{-1}, &p=i,\\
                       x_{\theta}, &p=0.
                     \end{cases}
\eeq
Thus
\beq\label{e:ttildinvwigen}
 {\tilde{t}}_{\varpi_i}\,(x_{\alpha}\otimes t^s) =(x_{\alpha}\otimes t^{s-(\varpi_i|\alpha)})\quad\textup{and}\quad {\tilde{t}}_{\varpi_i}^{-1}\,(x_{\alpha}\otimes t^s) =(x_{\alpha}\otimes t^{s+(\varpi_i|\alpha)}),\quad\forall\,\,\alpha\in R, \,s\in \mathbb{Z}.
 \eeq
\begin{proposition}\label{l0l1}
With notation as above, for $i\in I$, we have  $L(\Lambda_0)^{\tilde{t}_{\varpi_i}} \simeq L(\Lambda_i)$.
\end{proposition}
\begin{proof}
We consider the $\mathbf{U}(\afsl)$-linear map $L(\Lambda_i) \to L(\Lambda_0)^{{\tilde{t}_{\varpi_i}}}$ which sends $v_{\Lambda_i}$ to $v_{\Lambda_0}$.
To show this is well defined, we only need to check that $v_{\Lambda_0} \in  L(\Lambda_0)^{{\tilde{t}_{\varpi_i}}}$ satisfies the relations in \eqref{e:defofllambd1}-\eqref{e:defofllambd3} 
for $\Lambda = \Lambda_i$. 
Since $w_0w_{0,i}\Lambda_0=\Lambda_0$, the relations in  \eqref{e:defofllambd1} follows from \eqref{e:coincide1}--\eqref{e:coincide2}. 
The relations in \eqref{e:defofllambd2} are immediate from \eqref{e:ttildinvwi} by using Proposition~\ref{wtsbasicrep}~\eqref{i:weights}. Since $f_p\,v_{\Lambda_0}=0$ for $p\neq 0,i,$ and $x_{\theta}\,v_{\Lambda_0}=0,$ 
to prove the relations in \eqref{e:defofllambd3}, we only need to show that 
$ (f_i\otimes t^{-1})^2\,v_{\Lambda_0}=0$ in $L(\Lambda_0)$. But this follows easily by a standard $\mathfrak{sl}_2$ argument using the $\mathfrak{sl}_2$ copy spanned by 
$e_i\otimes t, f_i\otimes t^{-1}$, and $\alpha^\vee_{i}+c$.
Now, this map is a surjection, since $v_{\Lambda_0}$ generates $L(\Lambda_0)^{{\tilde{t}_{\varpi_i}}}$. Since
$L(\Lambda_i)$ is irreducible, it must be an isomorphism.
\end{proof}
For $i\in I$, let $T_{\varpi_i}$ be the isomorphism from $L(\Lambda_0)^{\tilde{t}_{\varpi_i}}$ onto $L(\Lambda_i)$. Observe that 
$$T_{\varpi_i}\,L(\Lambda_0)_\nu=L(\Lambda_i)_{t_{\varpi_i}(\nu)}, \qquad\forall\,\,\nu\in\widehat{\csa}^*.$$
Set $T_{-\varpi_i}:=T^{-1}_{\varpi_i}$.
The isomorphism $T_{-\varpi_i}: L(\Lambda_i) \to L(\Lambda_0)^{\tilde{t}_{\varpi_i}}$ maps $v_{\Lambda_i} \mapsto v_{\Lambda_0}$. 
It is then determined on all of $L(\Lambda_i)$ by $\afsl$-linearity, i.e., by the relation 
\beq\label{e:tmwi}T_{-\varpi_i}(x\,v) = \tilde{t}_{\varpi_i}^{-1}(x)\,T_{-\varpi_i}(v), \quad \forall\, \,  x \in \afsl,\, \ v \in L(\Lambda_i).\eeq
Now using \eqref{e:ttildinvwigen}, we get 
 \beq\label{e:tmwiprop}
 T_{-\varpi_i}\,(x^-_{\alpha}\otimes t^s)\, T_{\varpi_i}\,v=(x^-_{\alpha}\otimes t^{s-(\varpi_i|\alpha)})\,v,\quad\forall\,\,\alpha\in R^+,\, v\in L(\Lambda_0), \,s\in \mathbb{Z}.
 \eeq
\subsection{}\label{ss:tlamdef}
For $\lambda\in P^+$ and $\beta\in Q$, we define a linear isomorphism
$T_{\lambda-\beta}:L(\Lambda_0)\rightarrow L(\Lambda_{i_\lambda})$ as follows:
$$T_{\lambda-\beta}:=T_{\varpi_{i_\lambda}}T_{\lambda-\beta-\varpi_{i_\lambda}}.$$
Suppose there is $\lambda^\prime\in P^+$ and $\beta^\prime\in Q$ such that $\lambda-\beta=\lambda^\prime-\beta^\prime$. Then it is easy to see that $i_\lambda=i_{\lambda^\prime}$. Hence
the definition  is well defined. Observe that 
$$T_{\lambda-\beta}\,L(\Lambda_0)_\nu=L(\Lambda_{i_\lambda})_{t_{\lambda-\beta}(\nu)}, \qquad\forall\,\,\nu\in\widehat{\csa}^*.$$
Set $$T_{-(\lambda-\beta)}:=T^{-1}_{\lambda-\beta} \qquad \textup{and} \qquad \epsilon(\beta+w_{i_\lambda}, \,\beta'):=\epsilon(\beta,\,\beta'),\qquad \forall\,\, \lambda\in P^+, \,\beta, \beta'\in Q.$$
\begin{proposition}\label{p:tlambdambeta}Let $\lambda\in P^+$ and $\beta\in Q$. Then
\begin{enumerate}
 \item $T_{\lambda-\beta-d\alpha}\, T_{d\alpha}=\epsilon(\lambda-\beta-d\alpha,\,d\alpha)\,T_{\lambda-\beta},\quad\forall \,\,\alpha\in R^+, \, d\in \mathbb{Z}_{\geq0}.$
 \smallskip
 \item $T_{-(\lambda-\beta)}\,(x^-_{\alpha}\otimes t^s)\, T_{\lambda-\beta}\,v=(x^-_{\alpha}\otimes t^{s-(\lambda-\beta|\alpha)})\,v,\quad\forall\,\,\alpha\in R^+,\, v\in L(\Lambda_0),\, s\in \mathbb{Z}.$
\end{enumerate}
 \end{proposition}
 \begin{proof}
 The proof is immediate from \eqref{e:tmwiprop} and Proposition~\ref{fkprop}~\eqref{fkprop:p2}--\eqref{fkprop:p4}.
 \end{proof}
  
\subsection{}\label{ss:MTp}
Given $\lambda\in P^+, \,\pop\in\pnotl,$ $k\in\mathbb{Z}_{\geq0}$,  and $s\in I\cup\{r+1\}$, define $\epsilon_{\pop^k_s}\in\{\pm1\}$ as follows: 
$$\epsilon_{\pop^k_s}:=\prod_{p=s}^r\left( (-1)^{[\frac{d_{p,p}+k}{2}]}\prod_{j=p}^r \,\epsilon\big({\lambda+k\alpha_{1,p}-\sum_{p<i\leq j\leq r}d_{i,j}\alpha_{i,j}}-\sum_{u=j}^r d_{p,u}\alpha_{p,u}, \,\,d_{p,j}\alpha_{p,j}\big)\right).$$
Here, 
$[x]$ denotes the greatest integer less than or equal to $x$.

We are now in a position to state the main result of this section.
\begin{theorem}\label{MTp}
 Let $\lambda\in P^+,$ $\pop\in\pnotl,$ $k\in\mathbb{Z}_{\geq0}$, and $s\in I\cup\{r+1\}$. If $d_{\ell,\ell}(\pop)\geq\depthpl $ for all $s\leq \ell\leq r$, then 
 \beq\label{e:sMTp}
 \epsilon_{\pop^k_s}\,\rho_{\pop^k_s}\,T_{\lambda+k\theta}\,v_{\Lambda_0}= T_{\lambda+k\alpha_{1,s-1}-\sum_{s\leq i\leq j\leq r}d_{i,j}(\pop)\alpha_{i,j}}\, f_{\mathcal{I}(\pop_s)}\,v_{\Lambda_0},
 \eeq
 where  $f_{\mathcal{I}(\pop_s)}$ is a polynomial in $\var,$
 depends only  on the elements from the set $\mathcal{I}(\pop_s)$ such that the weight of  $f_{\mathcal{I}(\pop_s)}\,v_{\Lambda_0}$ in $L(\Lambda_0)$ is $\Lambda_0-\depthps\delta$.
\end{theorem}
%This theorem is proved in \S\ref{proof:MTp}.

\subsubsection{}\label{pf:MT}{\em Proof of Theorem~\ref{MT} from Theorem~\ref{MTp}.} 
We observe that  the expression on the left hand side of \eqref{e:sMTp} depends on $k$, the one on the right hand side, when $s=1$, is independent of it. 
The fact that these two expressions are equal when 
$d_{\ell,\ell}(\pop)\geq\depthpl$ for all $1\leq \ell\leq r$, what leads to the stability properties of interest. Thus Theorem~\ref{MTp} for $s=1$, proves Theorem \ref{MT}.

\medskip
The rest of the paper is devoted to proving Theorem~\ref{MTp}.
\subsection{} In this subsection, for $x\in \lieg$, $s\in\mathbb{Z}$, and  $m\in\mathbb{N},$ we set $xt^s:=x\otimes t^s$ and $[m]:=\{1,2,\ldots,m\}$. The following two lemmas are elementary.
\begin{lemma}\label{lem1} Let $\alpha\in R^+,\, y_\alpha\in {\lieg}_{-\alpha}, \, h_1,\ldots,  h_n\in \csa$,  and $p, q_1,\ldots, q_n\in\mathbb{Z}_{\geq0}$. 
Then
$$(y_\alpha  t^p)\,\big(\prod_{i=1}^n h_i  t^{-q_i}\big)=\sum_{0\leq k\leq n}\,\sum_{\substack{A\subseteq[n]\\|A|=k}}\big(\prod_{i\in A}\langle\alpha,\, h_i\rangle\big)\,\big(\prod_{i\in[n]\setminus A} h_i  t^{-q_i}\big)
\,\big(y_\alpha   t^{p-\sum_{i\in A} q_i}\big).$$
\end{lemma}
\begin{proof}
Proceed by induction on $n$. In case $n = 1$, we have  
$$(y_\alpha  t^p)\,( h_1  t^{-q_1})=( h_1  t^{-q_1})\,(y_\alpha  t^p)+[y_\alpha  t^p,\, h_1  t^{-q_1}]=( h_1  t^{-q_1})\,(y_\alpha  t^p)+\langle\alpha,\, h_1\rangle (y_\alpha   t^{p- q_1}),$$
and the result is obvious.
Now suppose that $n\geq 2$. Since $[y_\alpha  t^p,\, h_n  t^{-q_n}]=\langle\alpha, h_n\rangle (y_\alpha  t^{p-q_n})$, we have
 $$(y_\alpha  t^p)\,\big(\prod_{i=1}^n h_i  t^{-q_i}\big)=
 \big(( h_n  t^{-q_n})\,(y_\alpha  t^p)+\langle\alpha, h_n\rangle (y_\alpha  t^{p-q_n})\big)\,\big(\prod_{i=1}^{n-1} h_i  t^{-q_i}\big).$$
 Using the induction hypothesis the right hand side of the last equation becomes
\begin{align*}&( h_n  t^{-q_n})\sum_{0\leq k^\prime\leq n-1}\,\sum_{\substack{A^\prime\subseteq[n-1]\\|A^\prime|=k^\prime}}\big(\prod_{i\in A^\prime}\langle\alpha,\, h_i\rangle\big)\,\big(\prod_{i\in[n-1]\setminus A^\prime} h_i  t^{-q_i}\big)
\,\big(y_\alpha   t^{p-\sum_{i\in A^\prime} q_i}\big)\\
 &+
 \langle\alpha, h_n\rangle \sum_{0\leq k^{\prime\prime}\leq n-1}\,\sum_{\substack{A^{\prime\prime}\subseteq[n-1]\\|A^{\prime\prime}|=k^{\prime\prime}}}\big(\prod_{i\in A^{\prime\prime}}\langle\alpha,\, h_i\rangle\big)\,\big(\prod_{i\in[n-1]\setminus A^{\prime\prime}} h_i  t^{-q_i}\big)
\,\big(y_\alpha   t^{p-q_n-\sum_{i\in A^{\prime\prime}} q_i}\big).\end{align*}
This completes the proof. Indeed, for any $A\subseteq[n]$, there exists $B\subseteq [n-1]$ such that
either $A=B$ or $A=B\cup\{n\}$.
\end{proof}
\begin{lemma}\label{lem2}
Let $\alpha\in R^+,\,  y_\alpha\in {\lieg}_{-\alpha},\, h_1,\ldots,  h_n\in \csa$, and  $p_1, \ldots,  p_m, \,q_1,  \ldots,  q_n\in\mathbb{Z}_{\geq0}$. 
Then
\begin{align*}
&\big(\prod_{i=1}^m y_\alpha  t^{p_i}\big)\,\big(\prod_{i=1}^n h_i  t^{-q_i}\big)\\
&=\sum_{\substack{{{k_1,\ldots,  k_m}}\\0\leq k_i\leq n-\sum_{j=i+1}^{m}k_j}}\,\sum_{\substack{A_1,\ldots, A_m\\A_i\subseteq [n]\setminus \cup_{j=i+1}^{m} A_j\\|A_i|=k_i}}\big(\prod_{i\in \cup_{j=1}^m A_j}\langle\alpha,\, h_i\rangle\big)\,\big(\prod_{i\in[n]\setminus \cup_{j=1}^m A_j} h_i  t^{-q_i}\big)
\,\big(\prod_{j=1}^{m} y_\alpha   t^{p_j-\sum_{i\in A_j} q_i}\big).\end{align*}
\end{lemma}
\begin{proof}
 Proceed by induction on $m$. In case $m= 1$, we have the result from Lemma \ref{lem1}. Now suppose that $m\geq 2$. Using Lemma \ref{lem1}, we have
 \begin{equation}\label{pflem2e1}
  \begin{split}
  &\big(\prod_{i=1}^{m} y_\alpha  t^{p_i}\big)\,\big(\prod_{i=1}^n h_i  t^{-q_i}\big)\\
 &=\big(\prod_{i=1}^{m-1} y_\alpha  t^{p_i}\big)\sum_{0\leq k_m\leq n}\,\sum_{\substack{A_m\subseteq[n]\\|A_m|=k_m}}\big(\prod_{i\in A_m}\langle\alpha,\, h_i\rangle\big)\,\big(\prod_{i\in[n]\setminus A_m} h_i  t^{-q_i}\big)
\,\big(y_\alpha   t^{p_m-\sum_{i\in A_m} q_i}\big)\big).\end{split}\end{equation}                                                                                 
 Using the induction hypothesis,  we have
\begin{equation}\label{pflem2e2} \begin{split}&\big(\prod_{i=1}^{m-1} y_\alpha  t^{p_i}\big)\,\big(\prod_{i\in[n]\setminus A_m} h_i  t^{-q_i}\big)\\
&=\sum_{\substack{ k_1, \ldots, k_{m-1}\\0\leq k_i\leq n-\sum_{j=i+1}^{m}k_j}}\,\sum_{\substack{A_1, \ldots,  A_{m-1}
\\A_i\subseteq [n]\setminus \cup_{j=i+1}^m A_j\\|A_i|=k_i}}\big(\prod_{i\in \cup_{j=1}^{m-1} A_j}\langle\alpha,\, h_i\rangle\big)\,
\big(\prod_{i\in[n]\setminus \cup_{j=1}^m A_j} h_i  t^{-q_i}\big)
\,\big(\prod_{j=1}^{m-1} y_\alpha   t^{p_j-\sum_{i\in A_j} q_i}\big).\end{split}\end{equation}
Substituting \eqref{pflem2e2} in the right hand side of \eqref{pflem2e1}, we get the result.
\end{proof}
\subsection{}
The following result follows from \cite[Theorem 10 and Proposition 13 (1)]{RRV1}.
\begin{theorem}\cite{RRV1}\label{stabsl2}Let $\alpha\in R^+$. Let $d\in\mathbb{Z}_{\geq0}$ and $\underline{\pi}$ be a partition such that $d\geq |{\underline{\pi}}|$. Then
$$x^-_{\alpha}(d,d,\piseq)\,T_{d\alpha}\,v_{\Lambda_0}=(-1)^{[\frac{d}{2}]} f_{\underline{\pi}}\,v_{\Lambda_0},$$
where $f_{\underline{\pi}}$
is a polynomial in $\alpha^\vee t^{-j}, j\in\mathbb{N}$, depends only on ${\underline{\pi}}$ and not on $d$ such that the weight of $f_{\underline{\pi}}\,v_{\Lambda_0}$ 
in  $L(\Lambda_0)$ is  $\Lambda_0-|{\underline{\pi}}|\delta$.
 \end{theorem}

\begin{proposition}\label{crucprop} Let $d, d^\prime,m\in\mathbb{Z}_{\geq0},$ $\alpha\in R^+$, and $\underline{\pi}$ be a partition.  Let $\lambda\in P^+$ and $\beta\in Q$ with $(\lambda-\beta|\alpha)=d+d^\prime$, and set $\mu=\lambda-\beta$.
Let $\f_m$ be a polynomial in $\var,$  such that the weight of $\f_m\,v_{\Lambda_0}$ in $L(\Lambda_0)$ is $\Lambda_0-m\delta$. Then
\be
\item\label{pone:wt} the weight of $x^-_{\alpha}(d,d^\prime,\piseq)\,T_{\mu}\,\f_m\,v_{\Lambda_0}$ in $L(\Lambda_{i_\lambda})$ is $t_{\mu-d\alpha}\,(\Lambda_0-(|\underline{\pi}|+m)\delta)$. 
\item \label{ptwo:stab}
If $d\geq |\underline{\pi}|+m$, we have
 $$x^-_{\alpha}(d,d^\prime,\piseq)\,T_{\mu}\,\f_m\,v_{\Lambda_0}=(-1)^{[\frac{d}{2}]}\,\epsilon(\mu-d\alpha,\,d\alpha)\, T_{\mu-d\alpha}\,f_{\underline{\pi}, \f_m}\,v_{\Lambda_0},$$
 where $f_{\underline{\pi}, \f_m}$ is a polynomial in $\var,$ depends only on $\underline{\pi}$, $\f_m$ and not on $d, d^\prime$, such that 
 the weight of $f_{\underline{\pi}, \f_m}\,v_{\Lambda_0}$  in $L(\Lambda_{i_\lambda})$ is $\Lambda_0-(|\underline{\pi}|+m)\delta$.
\ee
 \end{proposition}
\begin{proof}
Since the weight of $T_{\mu}\,\f_m\,v_{\Lambda_0}$ is $t_\mu(\Lambda_0-m\delta)$ and $(\mu|\alpha)=d+d'$, we have the weight of
$x^-_{\alpha}(d,d^\prime,\piseq)\,T_{\mu}\,\f_m\,v_{\Lambda_0}$ is
$
t_\mu(\Lambda_0-m\delta)-d\alpha+(dd^\prime-|\piseq|)\delta
=t_{\mu-d\alpha}\,(\Lambda_0-(|\underline{\pi}|+m)\delta).
$ Hence part~\eqref{pone:wt}.
We now prove part~\eqref{ptwo:stab}. Using Propositions~\ref{fkprop} and  \ref{p:tlambdambeta}, we have 
\beq\begin{split}
     \big(\prod_{i=1}^{d}   x^-_\alpha\otimes t^{d^\prime-\pi_i}\big)\,T_{\mu}\,\f_m\,v_{\Lambda_0}
     &=\epsilon(\mu-d\alpha,\,d\alpha)\,T_{\mu-d\alpha}\big(\prod_{i=1}^{d}T_{-(\mu-d\alpha)} \, (x^-_\alpha\otimes t^{d^\prime-\pi_i})\,T_{\mu-d\alpha}\big)\,T_{d\alpha}\,\f_m\,v_{\Lambda_0}\\
     &=\epsilon(\mu-d\alpha,\,d\alpha)\,T_{\mu-d\alpha} \,\big(\prod_{i=1}^{d}   x^-_\alpha\otimes t^{d-\pi_i}\big)\,\f_m\,T_{d\alpha}\,v_{\Lambda_0}
    \end{split}
\eeq
 Using Lemma \ref{lem2},  the right hand side of the last equation becomes
$$\epsilon(\mu-d\alpha,\,d\alpha)\,T_{\mu-d\alpha} \,\sum_{q}f_{\f_m}^q\,\big(\prod_{i=1}^{d}   x^-_\alpha\otimes t^{d-\pi_i-\eta_{i,q}}\big)\,T_{d\alpha}\,v_{\Lambda_0},$$ 
for some polynomials $f_{\f_m}^q$ in $\var,$ and positive integers ${\eta_{i,q}}$, $1\leq i\leq d$, depend on $\f_m$ such that $m\geq \sum_{i=1}^d \eta_{i,q}$.
Part~\eqref{ptwo:stab} now follows from Theorem \ref{stabsl2} and part~ \eqref{pone:wt}.
\end{proof}

\begin{proposition}\label{pf:prop}
 Let  $\lambda\in P^+$,  $\pop\in\popset_\lambda,$  $k, m\in\mathbb{Z}_{\geq0}$, and $s\in I.$ Set $$\mu= {\lambda+k\alpha_{1,s}-\sum_{s< i\leq j\leq r}d_{i,j}\alpha_{i,j}}.$$
 Let $\f_m$ be a polynomial in $\var,$  such that the weight of $\f_m\,v_{\Lambda_0}$ in $L(\Lambda_0)$ is $\Lambda_0-m\delta$. Then for every $s< \q\leq r+1$, we have  
 \beq\begin{split}&\big(\prod_{j=\q}^r x_{s, j}^-(\dsj,\,\dprimesj+\delta_{1,s}k, \,\pijsseq)\big)
 \,T_{\mu}\,
 \f_m\,v_{\Lambda_0}\\&=\big(\prod_{j=\q}^r \epsilon(\mu-\sum_{u=j}^r d_{s,u}\alpha_{s,u},\, \dsj\alpha_{s,j})\big)\,
T_{\mu-
\sum_{j=\q}^r\dsj\,\alpha_{s,j}}\, f^{\q}_{g_m}\,v_{\Lambda_0},\end{split}\eeq
 where  $f^{\q}_{g_m}$ is a polynomial in $\var$,
 depends only  on $\f_m$ and the elements from the sets $\mathcal{I}^j_s(\pop), \,q\leq j\leq r$,
 such that the weight of $f^{\q}_{g_m}\,v_{\Lambda_0}$ in $L(\Lambda_0)$ is $\Lambda_0-\big(m+\sum_{j=\q}^r{d^j_s(\pop)}\big)\delta$.
\end{proposition}
\begin{proof}
 Proceed by induction on $\q$. In the case $\q=r+1$, by taking $f^{r+1}_{g_m}=\f_m$, both sides are equal to 
 $T_{\mu}\,
 \f_m\,v_{\Lambda_0}.$
 Now suppose that $\q\leq r$. By the induction hypothesis, we have
 \beq\label{pflemmaine1}\begin{split}&\big(\prod_{j=\q+1}^{r} x_{s, j}^-(\dsj,\dprimesj+\delta_{1,s}k, \pijsseq)\big)\,T_{\mu}\,
 \f_m\,v_{\Lambda_0}\\&=\big(\prod_{j=\q+1}^r \epsilon(\mu-\sum_{u=j}^r d_{s,u}\alpha_{s,u},\, \dsj\alpha_{s,j})\big)\,
T_{\mu-
\sum_{j=\q+1}^r\dsj\alpha_{s,j}}\, f^{\q+1}_{g_m}\,v_{\Lambda_0},\end{split}\eeq
where  $f^{\q+1}_{g_m}$ is a polynomial in $\var$,
 depends only  on $\f_m$ and the elements from the sets $\mathcal{I}^j_s(\pop), \,q< j\leq r$,
 such that the weight of $f^{\q+1}_{g_m}\,v_{\Lambda_0}$ in $L(\Lambda_0)$ is $\Lambda_0-\big(m+\sum_{j=\q+1}^r{d^j_s(\pop)}\big)\delta$.
 
 Acting both sides of \eqref{pflemmaine1} with $x_{s, \q}^-(\dsq,\,\dprimesq+\delta_{1,s}k, \,\pisqseq)$, we get
 \beq\label{pflemmaine2}\begin{split}&\big(\prod_{j=\q+1}^r \epsilon(\mu-\sum_{u=j}^r d_{s,u}\alpha_{s,u},\, \dsj\alpha_{s,j})\big)\,\big(\prod_{j=\q}^r x_{s, j}^-(\dsj,\,\dprimesj+\delta_{1,s}k,\, \pijsseq)\big)
 \,T_{\mu}\,
 \f_m\,v_{\Lambda_0}\\&=x_{s, \q}^-(\dsq,\,\dprimesq+\delta_{1,s}k,\, \pisqseq)\,T_{\mu-
\sum_{j=\q+1}^r\dsj\alpha_{s,j}}\, f^{\q+1}_{g_m}\,v_{\Lambda_0}.\end{split}\eeq
  Set $\nu= {\mu-
\sum_{j=\q}^r\dsj\alpha_{s,j}}$. 
We observe that
\begin{equation}\label{e:muofalphasq}
 \begin{split}
(\nu|\alpha_{s,\q})&=
(\lambda^{r+1}_s-\lambda^{r+1}_{\q+1})+\delta_{1,s}k+\sum_{j=\q+1}^r d_{\q+1,j}-\sum_{i=s+1}^{\q}d_{i,\q}-\sum_{j=\q}^r d_{s,j}-d_{s,\q}\\
&=(\lambda^{r+1}_s-\lambda^{r+1}_{\q+1})+\delta_{1,s}k+(\lambda^{r+1}_{\q+1}-\lambda^{\q+1}_{\q+1})-(\lambda^{\q}_{s}-\sum_{i=s}^{\q}
d^\prime_{i,\q}-\lambda^{\q+1}_{\q+1})-(\lambda^{r+1}_s-\lambda^{\q}_s)-d_{s,\q}\\
&=\sum_{i=s}^{\q} d^\prime_{i,\q}+\delta_{1,s}k-d_{s,\q}.
\end{split}\end{equation}  
Using Proposition \ref{p:tlambdambeta} and \eqref{e:muofalphasq}, the right hand side of \eqref{pflemmaine2} becomes
\beq\label{pflemmaine3}\begin{split}
    &\epsilon(\nu,\,\dsq\alpha_{s,q})\,T_\nu\,\big(\prod_{p=0}^{\dsq}T_{-\nu}\, \big(x^-_{s, \q}\otimes t^{d^\prime_{s, \q}+\delta_{1,s}k-\pi(\q)^{s}_p}\big)\,T_\nu \big)\,T_{\dsq \alpha_{s, \q}}\,f^{\q+1}_{g_m}\,v_{\Lambda_0}\\
 &=\epsilon(\nu,\,\dsq\alpha_{s,q})\,T_\nu\,\big(\prod_{p=0}^{d_{s,\q}}\big(x^-_{s, \q}\otimes t^{d_{s, \q}-\pi(\q)^{s}_p-\sum_{i=s+1}^{\q} d^\prime_{i,\q}}\big)\big)\,T_{d_{s, \q} \alpha_{s, \q}}\,f^{\q+1}_{g_m}\,v_{\Lambda_0}.
\end{split}\eeq
From Theorem \ref{weights}, it is easy to see that
\beq\label{pflemmaine5}
\big(\prod_{p=0}^{d_{s,\q}}\big(x^-_{s, \q}\otimes t^{d_{s, \q}-\pi(\q)^{s}_p-\sum_{i=s+1}^{\q} d^\prime_{i,\q}}\big)\big)\,T_{d_{s, \q} \alpha_{s, \q}}\,f^{\q+1}_{g_m}\,v_{\Lambda_0}=f^{\q}_{g_m}\,v_{\Lambda_0},\eeq
where $f^{\q}_{g_m}$ is a polynomial in $\var$,
 depends only  on $f^{\q+1}_{g_m}$ and the elements from the set $\mathcal{I}^q_s(\pop)$ such that the weight of $f^{\q}_{g_m}\,v_{\Lambda_0}$ in $L(\Lambda_0)$ is 
 $$\Lambda_0-\big(m+\sum_{j=\q+1}^r{d^j_s(\pop)}+d_{s,q}^2-   d_{s,q}(d_{s,q}-\sum_{i=s+1}^{q}d^\prime_{i,q}) +|\pisqseq|\big)\delta=\Lambda_0-\big(m+\sum_{j=\q}^r{d^j_s(\pop)}\big)\delta.$$ 
 Substituting \eqref{pflemmaine5} into \eqref{pflemmaine3}, we get the result.
\end{proof}

\subsection{}\label{proof:MTp}{\em Proof of Theorem \ref{MTp}.}
 Proceed by induction on $s$. In the case $s=r+1$, by taking $f_{\mathcal{I}({\pop_{r+1}})}=1$, both sides of \eqref{e:sMTp} are equal to $T_{\lambda+k\theta}\,v_{\Lambda_0}$. 
 Now suppose that $s\leq r$. 
 Set 
 $$\mu= {\lambda+k\alpha_{1,s}-\sum_{s< i\leq j\leq r}d_{i,j}\alpha_{i,j}}.$$
 By the induction hypothesis, we have
\beq\label{pf:eq1}
\rhocl{\pop^k_{s+1}} T_{\lambda+k\theta}\,v_{\Lambda_0}=\epsilon_{\pop^k_{s+1}}\, T_{\mu}\,f_{\mathcal{I}(\pop_{s+1})}\,v_{\Lambda_0},\eeq
where  $f_{\mathcal{I}(\pop_{s+1})}$ is a polynomial in $\var,$
 depends only  on the elements from the set $\mathcal{I}(\pop_{s+1})$ such that the weight of  $f_{\mathcal{I}(\pop_{s+1})}\,v_{\Lambda_0}$ in $L(\Lambda_0)$ is $\Lambda_0-\depthpspone\delta$.
Since $$\rho_{\pop^k_s}=\prod_{j=s}^r x_{s, j}^-(\dsj+\delta_{s,j}k,\,\dprimesj+\delta_{1,s}k, \,\pijsseq)\,\rhocl{\pop^k_{s+1}},$$ we get from \eqref{pf:eq1} that
$$\rhocl{\pop^k_{s}} \,T_{\lambda+k\theta}\,v_{\Lambda_0}=\epsilon_{\pop^k_{s+1}}\, \big(\prod_{j=s}^r x_{s, j}^-(\dsj+\delta_{s,j}k,\,\dprimesj+\delta_{1,s}k,\, \pijsseq)\big)\,T_{\mu}\,f_{\mathcal{I}(\pop_{s+1})}\,v_{\Lambda_0}.$$
Now using Proposition \ref{pf:prop} with $q=s+1$, we get 
\begin{align}&\epsilon_{\pop^k_{s+1}}\, (\prod_{j=s+1}^r \epsilon(\mu-\sum_{u=j}^r d_{s,u}\alpha_{s,u},\, d_{s,j}\alpha_{s,j}))\,\rhocl{\pop^k_{s}}\, T_{\lambda+k\theta}\,v_{\Lambda_0}\nonumber\\ 
&=x_{s, s}^-(\dss+k,\,\dprimess+\delta_{1,s}k, \,\pissseq)\,
T_{\mu-\sum_{j=s+1}^r d_{s,j}\alpha_{s,j}}\, f_{s+1}\,v_{\Lambda_0},\label{e:pfmtp}\end{align}
 where  $f_{s+1}$ is a polynomial in $\var,$
 depends only  on the elements from the sets $\mathcal{I}(\pop_{s+1})$ and $\mathcal{I}^j_s(\pop), \,s< j\leq r$, 
 such that the weight of $f_{s+1}\,v_{\Lambda_0}$ in $L(\Lambda_0)$ is $\Lambda_0-\big(\depthpspone+\sum_{j=s+1}^r{d^j_s(\pop)}\big)\delta$. 
 Since
 \begin{equation*}
 \begin{split}
 &d'_{s,s}+\delta_{1,s}k-\pi(s)^s_i-(\mu-\sum_{j=s+1}^r d_{s,j}\alpha_{s,j}-(d_{s,s}+k)\alpha_{s,s}\,|\,\alpha_{s,s})\\
 &=d'_{s,s}+\delta_{1,s}k-\pi(s)^s_i-(\lambda^{r+1}_s-\lambda^{r+1}_{s+1}+k+\delta_{1,s}k+\sum_{j=s+1}^rd_{s+1,j}-\sum_{j=s+1}^rd_{s,j}-2(d_{s,s}+k))\\
 &=d'_{s,s}-\pi(s)^s_i-(\lambda^{r+1}_s-\lambda^{r+1}_{s+1}-k-d_{s,s}+
 \sum_{j=s+1}^rd_{s+1,j}-\sum_{j=s}^rd_{s,j})\\
 &=d_{s,s}+k-\pi(s)^s_i,
 \end{split}\end{equation*}
 for all $1\leq i\leq s$, and
 $$d_{s,s}\geq d(\pop_s)=|\pissseq|+\depthpspone+\sum_{j=s+1}^r{d^j_s(\pop)},$$
 we get the result from \eqref{e:pfmtp} by using Proposition \ref{crucprop} and \eqref{e:invsetofps}.

\end{document}